\documentclass[11pt]{amsart}
\usepackage{mathrsfs}

\usepackage{latexsym}
\usepackage{amssymb}
\usepackage{amsfonts}

\usepackage{amscd}
\usepackage{bbm}
\usepackage{mathabx}
\usepackage{skak}

\usepackage{color}

\usepackage{hyperref}
\hypersetup{
    colorlinks,
    citecolor=black,
    filecolor=black,
    linkcolor=black,
    urlcolor=black
}

\usepackage[ps,dvips,color,all]{xy}

\newtheorem{proposition}{Proposition}[section]

\newtheorem{lemma}[proposition]{Lemma}
\newtheorem{definition}[proposition]{Definition}
\newtheorem{theorem}[proposition]{Theorem}

\newtheorem{corollary}[proposition]{Corollary}

\newtheorem{remark}[proposition]{Remark}

\newtheorem{lemma-definition}[proposition]{Lemma-Definition}

\newtheorem{conclusion}[proposition]{Conclusion}

\newcounter{tmp}

\def\coh{\operatorname{coh}}

\def\lto{\longrightarrow}

\def\D{{\mathcal D}}

\def\H{{\mathcal H}}
\def\K{{\mathcal K}}

\def\cO{{\mathcal O}}

\def\M{{\mathcal M}}
\def\cS{{\mathcal S}}
\def\cQ{{\mathcal Q}}

\def\T{{\mathcal T}}
\def\I{{\mathcal I}}

\def\ZZ{{\mathbb Z}}

\def\bR{{\mathbf R}}

\def\bL{{\mathbf L}}

\def\AA{{\mathbb A}}
\def\NN{{\mathbb N}}
\def\ZZ{{\mathbb Z}}

\def\PP{{\mathbb P}}

\def\Hom{\operatorname{Hom}}
\def\End{\operatorname{End}}
\def\Ad{\operatorname{Ad}}

\def\Ker{\operatorname{Ker}\,}

\def\Cone{\operatorname{Cone}}

\def\Spec{\operatorname{Spec}}

\def\id{{\operatorname{id}}}

\def\kk{{\Bbbk}}

\def\rad{\operatorname{rad}}
\def\op{\circ}

\newcommand{\Ho}{{\H^0}}

\newcommand{\SF}{\dS\!\dF\!\operatorname{--}\!}
\newcommand{\SFf}{\dS\!\dF_{fg}\!\operatorname{--}\!}
\newcommand{\prfdg}{\mathscr{P}\!\mathit{erf}\!\operatorname{--}}

\newcommand{\Ac}{\dA\!\mathit{c}\!\operatorname{--}\!}

\newcommand{\prf}{\operatorname{perf}\!\operatorname{--}}

\def\dA{\mathscr A}
\def\dB{\mathscr B}
\def\dC{\mathscr C}
\def\dD{\mathscr D}
\def\dE{\mathscr E}

\def\dF{\mathscr F}

\def\dM{\mathscr M}

\def\dR{\mathscr R}
\def\dS{\mathscr S}

\def\dX{\mathscr X}
\def\dY{\mathscr Y}

\def\Mod{{\mathscr M}\!\mathit{od}\!\operatorname{--}\!}

\def\mE{\mathsf E}
\def\mG{\mathsf G}
\def\mF{\mathsf F}
\def\mI{\mathsf I}
\def\mJ{\mathsf J}
\def\mK{\mathsf K}

\def\mM{\mathsf M}
\def\mN{\mathsf N}
\def\mP{\mathsf P}
\def\mQ{\mathsf Q}

\def\mT{\mathsf T}

\def\mf{\mathsf f}
\def\mi{\mathsf i}
\def\mj{\mathsf j}

\def\mp{\mathsf{p}}

\def\mt{\mathsf{t}}

\def\mPhi{\mathsf \Phi}
\def\mPsi{\mathsf \Psi}
\def\mPi{\mathsf \Pi}

\def\m0{\mathsf 0}

\def\rY{\mathrm{Y}}
\def\rS{\mathrm{S}}
\def\rT{\mathrm{T}}

\def\rs{\mathrm{s}}
\def\rt{\mathrm{t}}
\def\rv{\mathrm{v}}
\def\rw{\mathrm{w}}
\def\rf{\mathrm{f}}
\def\ri{\mathrm{i}}
\def\rj{\mathrm{j}}

\def\rp{\mathrm{p}}

\def\rPhi{\mathrm \Phi}
\def\rPsi{\mathrm \Psi}
\def\rPi{\mathrm \Pi}

\def\dHom{\mathsf{Hom}}
\def\dEnd{\mathsf{End}}

\def\wt{\widetilde}

\def\can{\operatorname{can}}
\def\Tot{\operatorname{Tot}}

\def\rd{{J}}
\def\rdi{\mJ_{-}}
\def\rde{\mJ_{+}}

{\endgroup\hfill$\Box$}

{\endgroup\hfill$\Box$}

\def\gR{R}

\def\gM{M}
\def\gN{N}
\def\gHom{\Hom}

\def\da{d_{\dA}}
\def\db{d_{\dB}}
\def\dc{d_{\dC}}

\def\dr{d_{\dR}}

\def\bone{{\mathbf 1}}
\def\btwo{{\mathbf 2}}

\def\mod{\operatorname{mod}\!}

\def\ev{{\operatorname{ev}}}
\def\Anm{\mathsf{A}}

\def\vrt{\mathtt{V}}
\def\Arr{\mbox{\LARGE{\it a}}}

\def\cX{\mathcal{X}}

\title[]{Twisted tensor product, smooth DG algebras, and noncommutative resolutions of singular curves
}

\author[]{Dmitri Orlov}

\address{ Algebraic Geometry Dept., Steklov Math. Institute RAS,
8 Gubkin str., Moscow 119991, RUSSIA}
\email{orlov@mi-ras.ru}

\date{}
\dedicatory{Dedicated to the blessed memory of Ludwig Dmitrievich Faddeev on the occasion of his 90th birthday}

\keywords{Noncommutative algebraic geometry, derived noncommutative schemes, differential graded algebras,  perfect modules and complexes}

\makeatletter
\@namedef{subjclassname@2020}{\textup{2020} Mathematics Subject Classification}
\makeatother
\subjclass[2020]{14A22, 16E45, 16P10, 16E35, 18G80}

\begin{document}

\begin{abstract}
New families of algebras and DG algebras with two simple modules are introduced and described. Using the twisted tensor product operation, we prove that such algebras have finite global dimension, and the resulting DG algebras are smooth.
This description allows us to show that some of these DG algebras determine smooth proper noncommutative curves that provide smooth minimal noncommutative resolutions for singular rational curves.
\end{abstract}

\maketitle


\section*{Introduction}

Noncommutative algebraic varieties not only arise naturally in solving various problems of algebraic geometry, but also provide new opportunities for the study of classical varieties and  establish new connections between algebraic objects and the geometric world.
By a derived noncommutative scheme $ \dX $ over a field $\kk$ we will mean a $\kk$\!--linear differential graded (DG) category which can be described as  the DG category perfect modules $ \prfdg \dR$ over a cohomologically bounded DG $\kk$\!--algebra $\dR,$
i.e. such a DG algebra that has only a finite number of nontrivial cohomology
(see , e.g., \cite{Or16,Or18,Or20}). The differential graded algebra $\dR$ depends on a choice of a classical generator $\mE\in\dX$ of the triangulated category of perfect modules  and it arises as the  DG algebra of endomorphisms
$\dEnd_{\dX} (\mE)$ of the object $\mE$ in the DG category $\dX.$

The most important properties of a DG algebra $\dR$ are those properties that are actually attributes of the DG category $\dX=\prfdg\dR.$ In this case, such properties can be considered and talked about as  properties of derived noncommutative schemes.
In this sense, one can define what it means for  a derived noncommutative scheme to be proper, smooth, or regular (see Definition \ref{propdef}).
Thus, for example, a derived noncommutative scheme $\dX=\prfdg\dR$ over $\kk$ will be called {\sf proper} if
the cohomology algebra
$\bigoplus_{p\in\ZZ}H^p(\dR)$
is finite-dimensional over $\kk.$

Another important concept arising in noncommutative geometry is a notion of  geometric realization of a  derived noncommutative scheme $\dX=\prfdg\dR$  which, by definition, is given by a full embedding of the triangulated category $\prf\dR$ into the triangulated category of perfect complexes $\prf X$ on a usual commutative scheme $X.$ Derived noncommutative schemes which have geometric realizations  will be called {\sf geometric}.
The most natural and interesting examples of geometric noncommutative schemes appear as admissible subcategories $\M \subset \prf X,$
where a scheme $X$ is smooth and projective. In this case, a derived noncommutative scheme arises as the DG subcategory $\dM \subset \prfdg X.$
Indeed, it can be shown that any such DG category $\dM$ has the form $\prfdg\dR,$ where $\dR$ is a cohomologically bounded DG algebra, and, moreover, the resulting derived noncommutative scheme $\dM$ is actually smooth and proper.

The most interesting and important derived noncommutative schemes for us are those that are geometric, proper and smooth at the same time.
For any two derived noncommutative schemes $ \dX $ and $ \dY $
 and any perfect bimodule $\mT$ we can define a new derived noncommutative scheme $\dX \underset{\mT} {\oright} \dY,$ which is called the gluing of these schemes via the bimodule $\mT.$
The noncommutative scheme $ \dX \underset {\mT} {\oright} \dY $ adopts many properties of the schemes $\dX$ and $\dY.$
For example, the smoothness and properness of $ \dX $ and $ \dY $ imply the smoothness and properness of the gluing
   $ \dX \underset {\mT} {\oright} \dY $ via a perfect bimodule $\mT.$
In \cite{Or16} it was proved that if noncommutative schemes $ \dX $ and $ \dY $ arise as admissible subcategories in categories of perfect complexes on smooth projective schemes, then their gluing $ \dX \underset {\mT} {\oright} \dY $ via a perfect bimodule $ \mT $ can be realized in the same way.

Another  natural class of derived noncommutative schemes is the class of such schemes  $\dX=\prfdg\dR,$ for which the DG algebra $\dR$ itself is finite-dimensional
over $\kk.$ In \cite{Or20} it was proved that this property is the property of the DG category $\prfdg\dR$ and for any other classical generator $\mE$ the DG
algebra $\dEnd_{\dX} (\mE)$ is quasi-isomorphic to a finite-dimensional DG algebra. Moreover, in \cite{Or19,Or20} it was shown  that any such noncommutative scheme has a geometric realization. In this way, we obtain a large class of derived noncommutative schemes which are proper and geometric. Although the most important of them are those that are also smooth (or regular).
For finite-dimensional algebras regularity is in fact equivalent to the finiteness of the global dimension with the additional separability property for the semisimple part if we are talking about smoothness.  It was shown that these properties can also be extended to DG algebras (see \cite{Or23}).
Algebras of finite global dimension may be naturally obtained from a directed quiver $Q$ with arbitrary relations $I.$
Any algebra of the form $A=\kk Q/I,$ where $\kk Q$ is the path algebra of a directed quiver $Q$ over a field $\kk,$ and $I$ is an ideal of relations, has  finite global dimension.
In this case, the DG category of perfect complexes $\prfdg A,$ which is smooth, can be obtained as a gluing of several categories of the form $\prfdg \kk$ via perfect bimodules. This demonstrates, using ordinary algebras as example, that the procedure of gluing smooth categories via perfect bimodules is one of the main operations for obtaining new smooth derived noncommutative schemes (see \cite{Or16,Or18,Or20}).

However, it is also very important to study more general algebras and DG algebras of finite global dimension, which, in particular,  cannot be obtained by a gluing procedure and for which categories of perfect modules do not have full exceptional collections.
Algebras of such  type with two simple modules appeared in \cite{G} and in \cite{KK}.
In this regard, the question also arises of finding other operations that allow one to construct smooth algebras and DG algebras, starting from known and elementary smooth algebras. In \cite{Or23} we considered an operation of twisted tensor product of noncommutative algebras and DG algebras over various subalgebras that do not belong to the center, and apply this operation to the construction of new smooth finite-dimensional algebras and DG algebras.
In that paper we gave sufficient conditions for a twisted tensor product of DG algebras to be  smooth.

In this paper we introduce and describe new families of smooth algebras with two simple modules.
We fix a vector space $C$ of dimension $n$ and  a non-increasing sequences of nonnegative integers $\overline{\mathfrak{k}}=\{k_1,\cdots, k_m\}$ such that $0\le k_m\le\cdots\le k_1 \le n.$
Denote by $\mathfrak{F}=\{V_1, W_1;\ldots; V_m, W_m\}$ a family of subspaces $V_i, \,W_i\subseteq C,\; i=1,\dots, m,$ of dimension $k_i$ and of  codimension $k_i,$ respectively. With any such family $\mathfrak{F}$ we associate an algebra $R_{\mathfrak{F}}$ and DG algebras $\dR_{\mathfrak{F}}(\chi),$ which are the algebra $R_{\mathfrak{F}}$ with different $\ZZ$\!--gradings induced by  homomorphisms $\chi:\ZZ^{m+1}\to \ZZ$ and with the zero differentials.
We show that all these DG algebras can be obtained by a sequence of twisted tensor products of elementary smooth algebras (see Theorem \ref{TwProd}).
Together with the results of the paper \cite{Or23}, this implies that under condition  $V_i\cap W_j=0$ when $i\ge j,$ which holds for subspaces $V_i,\,W_j$ in general position, the resulting DG algebras
$\dR_{\mathfrak{F}}(\chi)$ are smooth (see Theorem \ref{smoothtwist}). This construction gives a large number of families of smooth algebras, which include both examples from the paper \cite{G} and examples from the paper \cite{KK}. In Section \ref{EqdimF} we  consider in more detail DG algebras $\dR_{\mathfrak F}(\chi)$ that are related to equidimensional families ${\mathfrak F},$ i.e. such families that all $k_i$ equal to the same integer $k.$
In particular, we show that there is a full embedding of the triangulated category $\prf\dR_{\mathfrak F}(\chi)$ to a triangulated category  with an exceptional collection of length $m+2,$ which in fact is a three-block exceptional collection of the type $(1,1,m)$ (see Theorem \ref{imbedR}). This description allows us to show that in the case of two-dimensional vector space $C$ and one-dimensional subspaces $V_i, W_j,$ the DG algebras
$\dR_{\mathfrak F}(\chi_0)$ for special grading $\chi_0$ provide smooth noncommutative resolutions for singular rational curves (see Theorems \ref{nodres}, \ref{genres}). We also prove that the dimension of a triangulated
categories $\prf\dR_{\mathfrak F}(\chi)$ is equal to $1$ in such case (see Theorem \ref{dimcat}). Thus, in the case of  equidimensional families $\mathfrak{F}$   with $n=2$ and $k=1,$ the categories $\prf\dR_{\mathfrak{F}}(\chi)$
can be considered as  smooth proper noncommutative curves.
Moreover, for  some special families $\mathfrak{F}$ and gradings $\chi_0$ the categories $\prf\dR_{\mathfrak{F}}(\chi_0)$ give  smooth noncommutative resolutions for singular rational curves $\rY_{\mathfrak{F}},$
which can be directly constructed by the family $\mathfrak{F}.$
Such noncommutative resolutions are minimal in sense that the whole K-theory of the categories $\prf\dR_{\mathfrak{F}}(\chi)$ is isomorphic to the direct sum $K_*(\kk)\oplus K_*(\kk)$ of two copies of the K-theory of the base field.

The author is very grateful to Anton Fonarev and Alexander Kuznetsov for useful discussions and valuable comments.

\section{Preliminaries}

\subsection{Differential graded algebras and modules}

Let $\kk$  be a field. Recall that  a {\sf differential graded $\kk$\!--algebra (=DG
algebra)} $\dR=(\gR, \dr)$ is  a $\ZZ$\!--graded associative $\kk$\!--algebra
$
\gR =\bigoplus_{q\in \ZZ} R^q
$
endowed with a $\kk$\!--linear differential $\dr: \gR \to \gR$  (i.e. a homogeneous
map $\dr$ of degree 1 with $\dr^2 = 0$) that satisfies the graded Leibniz rule.
We consider DG algebras with an identity element $1\in R^0.$ In this case, $\dr(1)=0.$

A {\sf differential graded module $\mM$ over  $\dR$ (=DG $\dR$\!--module)} is a $\ZZ$\!--graded right
$\gR$\!--module
$
\gM = \bigoplus_{q\in\ZZ} M^q
$
endowed with a $\kk$\!--linear differential $d_{\mM}: \gM \to \gM$ of degree 1,  $d_{\mM}^2=0,$ satisfying
the graded Leibniz rule
$
d_{\mM}(mr) = d_{\mM}(m) r + (-1)^q m \dr( r)
$
for all $ m\in M^q, r\in \gR.$

For any two DG modules $\mM$ and $\mN$ we define a complex of $\kk$\!--vector spaces $\dHom_{\dR} (\mM, \mN)$
as the graded vector space
$\gHom_{\gR}^{gr} (\gM, \gN):=\bigoplus_{q\in\ZZ}\Hom_{\gR} (\gM, \gN)^q,$
where $\Hom_{\gR} (\gM, \gN)^q$ is the space of homogeneous homomorphisms of $\gR$\!--modules of degree $q.$
The differential $D$ of the complex $\dHom_{\dR} (\mM, \mN)$ is defined by the rule
$
D(f) = d_{\mN} \circ f - (-1)^q
f\circ d_{\mM}$
for any $ f\in \Hom_{\gR} (\gM, \gN)^q.$

All (right) DG $\dR$\!--modules form a DG category
$\Mod \dR$ (see \cite{Ke}). Let $\Ac\dR$ be the full
DG subcategory consisting of all acyclic DG modules, i.e. DG modules with trivial cohomology.
The
homotopy category $\Ho(\Mod\dR)$ has a natural structure of a triangulated category,
and the homotopy subcategory of acyclic complexes $\Ho (\Ac\dR)$ forms a full triangulated subcategory in it.
The {\sf derived
category} $\D(\dR)$ is defined as the Verdier quotient
\[
\D(\dR):=\Ho(\Mod\dR)/\Ho (\Ac\dR).
\]

The derived category $\D(\dR)$ is equivalent to the homotopy category $\Ho(\SF\dR),$
where $\SF\dR\subset\Mod\dR$ is the DG subcategory of all semi-free modules.
Recall that a DG module
$\mP$ is called {\sf semi-free} if it has a filtration
$0=\mPhi_0\subset \mPhi_1\subset ...=\mP=\bigcup \mPhi_n$
with free quotients  $\mPhi_{i+1}/\mPhi_i$ (see \cite{Ke}).
We will also need a notion of semi-projective module.

\begin{definition}\label{spr} A DG $\dR$\!--module $\mM$ is called {\sf semi-projective (DG projective)} if the following equivalent conditions hold:
\begin{itemize}
\item[1)] the DG functor $\dHom_{\dR}(\mM, -)$ preserves surjective quasi-isomorphisms,
\item[2)] $\mM$ is projective as an $R$\!--module and $\dHom_{\dR}(\mM, -)$ preserves quasi-isomorphisms,
\item[3)] $\mM$ is a direct summand of some semi-free DG $\dR$\!--module.
\end{itemize}
\end{definition}

Denote by $\SFf\dR\subset \SF\dR$ the full DG subcategory of finitely generated semi-free
DG modules, i.e. such semi-free DG modules that $\mPhi_n=\mP$ for some $n,$ and $\mPhi_{i+1}/\mPhi_i$ is a finite direct sum of
$\dR[m].$
The {\sf DG category of perfect modules} $\prfdg\dR$
is the full DG subcategory of $\SF\dR$ consisting of all DG modules that are isomorphic to direct summands of objects of $\SFf\dR$
in the homotopy category $\Ho(\SF\dR).$
The homotopy category $\Ho(\prfdg\dR),$ which we denote by $\prf\dR,$  is called {\sf the triangulated category of perfect modules.}
In other words, the category $\prf\dR$ is the subcategory in $\D(\dR)$  that is (classically) generated by the DG algebra $\dR$ itself in the following sense.

\begin{definition}
A set $S$ of objects  of a triangulated category $\T$ {\sf (classically) generates} $\T$
if the smallest full triangulated subcategory of
$\T$ containing $S$ and  closed under taking direct summands coincides with the whole category $\T.$
In the case where the set $S$ consists of a single object $E\in \T,$
the object $E$ is called a {\sf classical generator} of $\T.$
\end{definition}

A classical generator, which  generates a triangulated category in a finite number of steps, is called
a {\sf strong generator}.
More precisely, let $\I_1$ and $\I_2$ be two full subcategories of a triangulated category $\T.$ Denote by $\I_1*\I_2$ the full subcategory of $\T$
consisting of all objects $M$ for which there is an exact triangle $M_1\to M\to M_2$ with $M_i\in \I_i.$
For any subcategory $\I\subset\T$ denote by $\langle \I\rangle$ the smallest full subcategory of $\T$ containing $\I$ and closed under
taking finite direct sums, direct summands, and shifts. We put $\I_1 \diamond\I_2=\langle \I_1*\I_2\rangle$ and  define by induction
$\langle \I\rangle_k=\langle\I\rangle_{k-1}\diamond\langle \I\rangle.$ If $\I$ consists of a single object $E,$ we denote $\langle \I\rangle$ as
 $\langle E\rangle_1$ and put by induction $\langle E\rangle _k=\langle E\rangle_{k-1}\diamond\langle E\rangle_1.$
\begin{definition}
An object $E\in\T$ is called  a {\sf strong generator} if $\langle E\rangle_n=\T$ for some $n\in\NN.$
\end{definition}

If a triangulated category has a strong generator, then all of
its classical generators are  also strong.

Following \cite{Ro}, we can define a dimension for a triangulated
category with a strong generator.
\begin{definition}
The {\em dimension} of a triangulated category $\T,$ denoted by $\dim \T,$ is the smallest integer $d\ge 0$ such that there exists an object $E\in \T$
for which $\langle X\rangle_{d+1}=\T.$
\end{definition}

Let us now discuss some basic properties of DG algebras and DG and triangulated categories of perfect modules over DG algebras.
\begin{definition}\label{propdef} Let $\dR$ be a DG $\kk$\!--algebra. Then
\begin{itemize}
\item[(1)] $\dR$ is called {\sf proper} if the cohomology algebra $\bigoplus_{p\in\ZZ}H^p(\dR)$ is finite-dimensional.
\item[(2)] $\dR$ is called {\sf $\kk$\!--smooth} if  it is perfect as a DG bimodule, i.e. as a DG module over the DG algebra $\dR^{\op}\otimes_{\kk}\dR.$
\item[(3)] $\dR$ is called {\sf regular} if $\dR$ is a strong generator for  the triangulated category $\prf\dR.$
\end{itemize}
\end{definition}

All these properties are properties of the DG category $\prfdg\dR.$ It is easy to see that $\dR$ is proper if and only if
$\bigoplus_{m\in\ZZ}\Hom(X, Y[m])$ is finite-dimensional for any two objects $X, Y\in\prf\dR.$  It is also known that
any smooth DG algebra is regular (see \cite{Lu}).

\begin{definition} For any regular DG algebra $\dR$ we can define a global dimension of $\dR$ as
the smallest integer $d\ge 0$
for which $\langle \dR\rangle_{d+1}=\prf\dR.$
\end{definition}

This definition is consistent with the standard definition of the global dimension for algebras, as follows from Lemma 2.5 of \cite{KrK} (see also  \cite[Th.8.3]{Cr}).

\subsection{Twisted tensor product of algebras and DG algebras}

Recall the definition of a twisted tensor product of algebras and DG algebras.
Let $R, A, B$ be  $\kk$\!--algebras and  $\epsilon_A: R\to A$ and $\epsilon_B: R\to B$ be morphisms of the algebras.
In such case we say that $A$ and $B$ are $R$\!--rings.

\begin{definition}
A {\sf twisted tensor product over $R$} of two $R$\!--rings $A$ and $B$
is an $R$\!--ring $C$ together with two $R$\!--rings morphisms $i_A :
A\to C$ and $i_B : B \to C$ such that the canonical  map $\phi: A\otimes_R B \to C$
defined by $\phi(a\otimes b) := i_A(a)\cdot i_B(b)$ is an isomorphism of $R$\!--bimodules.
\end{definition}

There is a direct way to describe  twisted tensor products.
Let $\phi : A\otimes_R B \stackrel{\sim}{\to} C$ be the canonical isomorphism used in the definition
of the twisted tensor product. Then, we can define a map
\[
\tau : B\otimes_R A \to A\otimes_R B\quad \text{by the rule}\quad
\tau (b \otimes a) :=
\phi^{-1}(i_B(b) \cdot i_A(a)).
\]

Conversely, let $\tau: B\otimes_R A\to A\otimes_R B$ be an $R$\!--bilinear map for which
\begin{equation}\label{fixsides}
\tau (1\otimes a)= a\otimes 1,\; \tau(b\otimes 1)=1\otimes b.
\end{equation}
In this case,  we can  define a multiplication $\mu_{\tau}:=(\mu_A\otimes \mu_B)\circ (1\otimes\tau\otimes 1)$ on the $R$\!--bimodule $A\otimes_R B.$
The multiplication $\mu_{\tau}$ is associative if and only if there is an equality
\begin{equation}\label{twist}
\tau\circ (\mu_B\otimes \mu_A) = (\mu_A\otimes \mu_B) \circ (1\otimes\tau\otimes 1) \circ (\tau \otimes \tau ) \circ (1\otimes\tau\otimes 1)
\end{equation}
of maps from $B\otimes_R B \otimes_R  A \otimes_R A$ to $A \otimes_R B$ (see, e.g., \cite{CSV}).

\begin{definition}
An $R$\!--bilinear map  $\tau$ that satisfies conditions (\ref{fixsides}) and (\ref{twist}) is called a {\sf twisting map} for $A$ and $B$ over $R,$
and we denote the $R$\!--ring $(A\otimes_R B, \mu_{\tau})$ by $A\otimes_R^{\tau}B.$
\end{definition}

The notion of a twisted tensor product can be easily extended to the case of DG algebras.
Let $\dR, \dA, \dB$ be  DG $\kk$\!--algebras and  $\epsilon_{\dA}: \dR\to \dA$ and $\epsilon_{\dB}: \dR\to \dB$ be morphisms of DG algebras,
i.e. the DG algebras $\dA$ and $\dB$ are DG $\dR$\!--rings (or DG rings over $\dR$).

\begin{definition}\label{TwistDGal}
A {\sf twisted tensor product over $\dR$} of two DG $\dR$\!--rings $\dA$ and $\dB$
is a DG $\dR$\!--ring $\dC$ together with two $\dR$\!--rings morphisms $i_{\dA} :
\dA\to \dC$ and $i_{\dB} : \dB \to \dC$ such that the canonical  map $\phi: \dA\otimes_{\dR} \dB \to \dC$
defined by $\phi(a\otimes b) := i_{\dA}(a)\cdot i_{\dB}(b)$ is an isomorphism of $\dR$\!--bimodules.
\end{definition}

It follows from the definition that the differential of $\dC$ is uniquely determined by the Leibniz rule, because we have
$\dc(a\otimes b)=\da(a)\otimes b+(-1)^{\deg(a)} a\otimes \db(b).$
The twisting map $\tau$ associated with a twisted tensor product of DG rings satisfies conditions (\ref{fixsides}) and (\ref{twist}) and, additionally,
\begin{equation}
\tau(\db(b)\otimes a) + (-1)^{\deg(b)}\tau( b\otimes \db(a)) = \dc( \tau(b\otimes a)),
\end{equation}
which means that the map $\tau: \dB\otimes_{\dR} \dA\to \dA\otimes_{\dR} \dB$ should be a map of DG $\dR$\!--bimodules.

\begin{definition}
As for algebras we denote the DG $\dR$\!--ring $\dC=(\dA\otimes_{\dR} \dB, \mu_{\tau}, \dc)$ by $\dA\otimes_{\dR}^{\tau}\dB.$
\end{definition}

\begin{remark}{\rm
Note that for DG algebras there is a generalization of the twisted tensor product which is called a DG twisted tensor product and was introduced in \cite{Or23} (Definition 2.15).
}
\end{remark}

Suppose that the $\dR$\!--rings $\dA$ and $\dB$ have  $\dR$\!--augmentations, i.e. there are given morphisms $\pi_{\dA}: \dA\to \dR$ and $\pi_{\dB}: \dB\to \dR$ for which compositions  $\pi_{\dA}\circ\epsilon_{\dA}$
and $\pi_{\dB}\circ\epsilon_{\dB}$ are the identity maps.
Denote by $\mI_{\dA}=\Ker\pi_{\dA} \subset \dA$ and $\mI_{\dB}=\Ker\pi_{\dB} \subset \dB$ the augmentation DG ideals.
In such a case we have a special  twisting map  $\mathbf{v}: \dB\otimes_{\dR} \dA\to \dA\otimes_{\dR} \dB$ that is defined by the following rule
\begin{equation}\label{vtwist}
\mathbf{v}(b\otimes a)=\epsilon_{\dA}(\pi_{\dB}(b))\cdot a\otimes 1 + 1\otimes b\cdot \epsilon_{\dB}(\pi_{\dA}(a))-\epsilon_{\dA}(\pi_{\dB}(b))\otimes\epsilon_{\dB}(\pi_{\dA}(a)).
\end{equation}
For this twisted tensor product, we obtain $(1\otimes b)(a\otimes 1)=\mathbf{v}(b\otimes a)=0,$ whenever $a\in \mI_{\dA}, b\in \mI_{\dB}.$

Let $\dR=(\gR, \dr)$ be a finite-dimensional DG algebra over a base field $\kk.$
Denote by $\rd\subset R$ the (Jacobson) radical of the $\kk$-algebra $R.$ It is nilpotent and it is a graded ideal. However, the radical $\rd$ is not necessary a DG ideal.
In fact, with any two-sided graded ideal $I\subset\gR$ we can associate two DG ideals $\mI_{-}$ and $\mI_{+},$
where an {\sf internal} DG ideal $\mI_{-}=(I_{-}, \dr)$ consists of all $r\in I$ such that $\dr( r)\in I,$ while
an {\sf external} DG ideal  $\mI_{+}=(I_{+}, \dr)$  is the sum $I+\dr (I).$
Let $\rd\subset\gR$ be the radical of $\gR.$ The DG ideals $\rdi$ and $\rde$ will be called {\sf internal} and {\sf external} DG radicals of DGA $\dR.$
It is known that the natural homomorphism of DGAs $\dR/\rdi\to \dR/\rde$ is a quasi-isomorphism \cite[Lemma 2.4]{Or20}.
The following theorem is a particular case of Theorem 3.13 from \cite{Or23}.

\begin{theorem}\label{DGfinpr}
Let $\dR$ be a finite-dimensional DG algebra. Let $\dA$ and $\dB$ be finite-dimensional $\dR$\!--rings  such that $\dA$ has an augmentation $\pi_A: \dA\to \dR$ with an ideal $\mI_{\dA}=\Ker\pi_{\dA}$ and $\dB$ is semi-projective as a left DG $\dR$\!--module.
Let $\dC=\dA\otimes_{\dR}^{\tau}\dB$ be a twisted tensor product.
 Assume that  $\dA, \dB$ are smooth and the ideal $\mI_{\dA}\otimes_{\dR}\dB\subset \dC$ is contained in the external radical $(\mJ_{\dC})_{+}.$
Then, the DG algebra $\dC=\dA\otimes_{\dR}^{\tau}\dB$  is also smooth.
\end{theorem}

\section{Algebras and DG algebras with two simple modules}

\subsection{Quivers with two vertices and Kronecker algebras}

Let $Q$ be a quiver. It consists of the data
$(\vrt, \Arr, s, t),$
where $\vrt, \Arr$ are finite sets of vertices and arrows, respectively, while
$s, t : \Arr\to  \vrt$
are maps associating to each arrow its source and target.
Let $\kk$ be a field, the path $\kk$\!--algebra $\kk Q$ is  determined
by the generators $e_q$ for $q\in \vrt$ and $a$ for $a\in \Arr$
with the following  relations:
$
e^2_q = e_q,\quad e_r e_q = 0,$
 when
 $r\ne q,$
and
$
e_{t(a)} a = a e_{s(a)} = a.
$
As a $\kk$\!--vector space,
the path algebra $\kk Q$ has a basis consisting of the set of all paths in $Q,$
where a path $\overline{p}$ is a possibly empty sequence $a_{l} a_{l-1}\cdots a_1$ of compatible arrows,
i.e. $s(a_{i+1})=t(a_i)$ for all $i=1,\dots, l-1.$
For an empty path we have to choose a vertex of the quiver.
The composition of two paths $\overline{p}_1$ and $\overline{p}_2$
is defined as the path $\overline{p}_2 \overline{p}_1$ if they are compatible
and as $0$ if they are not compatible.

To obtain a more general class of algebras, we need to introduce quivers with
relations.
A relation on a quiver $Q$ is a subspace of $\kk Q$ spanned by linear
combinations of paths having a common source and a common target, and of length at
least 2.
A quiver with relations is a pair $(Q, I),$ where $Q$ is a quiver and $I$ is a two-sided ideal
of the path algebra $\kk Q$ generated by relations.
The quotient algebra $\kk Q/I$ will be called the quiver algebra of the quiver with relations
$(Q, I).$

Recall that a finite-dimensional $\kk$\!--algebra $\Lambda$ is called basic, if the semisimple part $\Lambda/\rad(\Lambda)$ is isomorphic to $\kk\times\cdots\times \kk.$
It is well-know that any basic finite-dimensional algebra $\Lambda$ is isomorphic to an algebra $\kk Q/I$ for some
quiver with relations $(Q, I)$ (see \cite{Ga}). Moreover, the number of vertices of the quiver $Q$ coincides with the number of simple $\Lambda$\!--modules.

Let $Q_{n,m}$ be a quiver with two vertices $\bone, \btwo$ and with $n$ arrows $\{c_1,\ldots, c_n\}$  from $\bone$ to $\btwo$ and $m$ arrows $\{b_1,\ldots, b_m\}$
 from  $\btwo$ to $\bone.$
\begin{equation}\label{TwoVQ}
Q_{n,m}=\Bigl[
\xymatrix{
\underset{\btwo}{\bullet}\ar@/_0.5pc/[rr]_{b_1} \ar@{{ }{ }}@/_1pc/[rr]_{\vdots} \ar@/_3pc/[rr]_{b_m} & &
\underset{\bone}{\bullet}\ar@/_0.5pc/[ll]_{c_1}  \ar@{{ }{ }}@/_1.5pc/[ll]_{\vdots} \ar@/_3pc/[ll]_{c_n}
}
\Bigr].
\end{equation}

Let $\kk Q_{n,m}$ be the path algebra of the quiver $Q_{n,m}.$
Denote by $B$ and $C$ the vector spaces generated by the arrows $\{b_1,\ldots, b_m\}$ and $\{c_1,\ldots, c_n\},$ respectively.
\begin{remark} {\rm
We do not consider quivers with loops, because quiver algebras with relations for quivers with loops have infinite global dimensions
(see \cite[Cor. 5.6]{Ig}).
}
\end{remark}

The quivers $Q_{n,0}$ are called Kronecker quivers and their path algebras $K_n=K(C)=\kk Q_{n,0}$ are called Kronecker algebras.
Any embedding  of vector spaces $V\subseteq C$ induces an embedding  of the Kronecker algebras $K(V)\subseteq K_n=K(C).$
Denote by $K_m^{\op}$ the Kronecker algebra $\kk Q_{0,m}.$ This algebra is opposite to the algebra $K_m=\kk Q_{m,0}$ on the one hand, and on the other hand it is obviously isomorphic to the algebra $K_m$ as well.

\subsection{New examples of smooth DG algebras}
In this section we introduce and describe new families of smooth algebras with two simple modules.
For simplicity, we assume that the base field $\kk$ is infinite.

We start with a quiver of the form $Q_{n,m},$ as in (\ref{TwoVQ}), and fix some new relations.
These relations will also depend on some sequences of integers $\{k_1,\cdots, k_m\}$ for which $0\le k_m\le\cdots\le k_1 \le n.$

Let us fix a sequence of integers $\overline{\mathfrak{k}}=\{k_1,\cdots, k_m\}$ such that $0\le k_m\le\cdots\le k_1 \le n.$ Let $V_i\subseteq C,\; i=1,\dots,m,$ be some subspaces of dimension $k_i$ and $W_i\subseteq C,\; i=1,\dots, m,$ be some subspaces of codimension $k_i.$
Denote by $\mathfrak{F}=\{V_1, W_1;\ldots; V_m, W_m\}$ the resulting family of subspaces.
We will assume that the family $\mathfrak{F}$ satisfies the following property:
\[
 V_i\cap W_j=0 \quad \text{for any pair} \quad i\ge j. \tag{G}
\]

\noindent
Note that $\dim V_i+\dim W_j=n-k_j+k_i\le n,$ when $i\ge j.$ Therefore, the property (G) holds for subspaces $V_i$ and $W_j$ in  general position.

\begin{definition}\label{defalg} Let $\mathfrak{F}=\{V_1, W_1;\ldots; V_m, W_m\}$ be a family of subspaces
satisfying  Property (G).
Define an algebra $R_{\mathfrak{F}}$ as the quotient algebra $R_{\mathfrak{F}}=\kk Q_{n,m}/I_{\mathfrak{F}},$
where $I_{\mathfrak{F}}$ is an ideal generated by the following elements:
\begin{flalign*}
\begin{array}{rlll}
\mathrm{1)} & b_i c b_j  & \text{for any}\quad c\in C,& \text{when}\quad i\le j;\\
\mathrm{2)} & b_i v & \text{for any}\quad v\in V_i,& \text{where}\quad i=1,\dots, m;\\
\mathrm{3)} & w b_i & \text{for any}\quad w\in W_i,& \text{where}\quad i=1,\dots, m.
\end{array}&&
\end{flalign*}
\end{definition}
\begin{remark}
{\rm
In the case $\mathfrak{F}=\emptyset,$ the algebra $R_{\emptyset}$ is isomorphic to the Kronecker algebra $K_n=K(C)=\kk Q_{n,0}.$
}
\end{remark}

It is easy to see that the algebra $R_{\mathfrak{F}}$ is finite dimensional.
Let $J\subset \kk Q_{n,m}$ be the ideal generated by relation 1) of Definition \ref{defalg}. The quotient algebra $\Anm=\kk Q_{n,m}/J$ is finite-dimensional and, for example, the vector space
$e_1 \Anm e_2$ is generated by elements of the form $b_{p_s}v_{p_s p_{s-1}} \cdots  v_{p_2 p_1}b_{p_1},$ where $1\le s\le m,\; \{p_1,\dots, p_s\}\subseteq \{1,\dots, m\}$ and
$v_{ij}\in C$ for all $i,j.$ For any family ${\mathfrak F}$  the algebra $R_{\mathfrak F}$ can be presented as the  quotient algebra $\Anm/\overline{I}_{\mathfrak F},$ where $\overline{I}_{\mathfrak F}$ is the twosided ideal of the algebra $\Anm$ generated by  relations 2) and 3)
of Definition \ref{defalg}.

The algebra $\Anm=\kk Q_{n,m}/J$ can be graded by the free abelian group $\ZZ^{m+1}.$ Let $\varepsilon_i\in \ZZ^{m+1},\; i=0,1,\dots, m,$ be a basis of the lattice $\ZZ^{m+1}.$ A $\ZZ^{m+1}$\!--grading on the algebra  $\Anm$ is given by setting $\deg(v)=\varepsilon_0$ for all $v\in C$ and $\deg (b_i)=\varepsilon_i$ for any $i=1,\dots, m.$ The nontrivial components of this $\ZZ^{m+1}$\!--grading of the algebra $\Anm$ are related to the pairs $(u;P),$ where  $P=\{p_1,\dots, p_s\}$ is a subset of $\{1,\dots, m\}$ and $u$ is a nonnegative integer that can be equal to
$s-1,\, s,$ or $s+1.$ Thus, the pair $(u;P)$ corresponds to the element $u\varepsilon_0+\varepsilon_{p_1}+\cdots + \varepsilon_{p_s}\in \ZZ^{m+1}.$

It is easy to see that  the graded component $\Anm_{(s-1;P)}$ is isomorphic to the vector subspace
$b_{p_s}C b_{p_{s-1}}\cdots b_{p_2} C b_{p_1}$ that is generated by all elements of the form  $b_{p_s}v_{p_s p_{s-1}} \cdots  v_{p_2 p_1}b_{p_1}$ with $v_{ij}\in C.$
The component $\Anm_{(s+1;P)}$ is isomorphic to the subspace $C b_{p_s}\cdots  b_{p_1} C,$ while the component $\Anm_{(s;P)}$ is the direct sum of the vector subspaces $C b_{p_s} \cdots C b_{p_1}$ and $b_{p_s}C \cdots  b_{p_1} C.$ In particular, we have $\Anm_0=\Anm_{(0;\emptyset)}=\langle e_1\rangle \oplus\langle e_2\rangle,\; \Anm_{\varepsilon_0}=\Anm_{(1;\emptyset)}=C$ and $\Anm_{\varepsilon_i}=\Anm_{(0; \{i\})}\cong\langle b_i\rangle$ for any $i=1,\dots, m.$

Since the relations 2) and 3) of Definition \ref{defalg} are homogenious, the ideal $\overline{I}_{\mathfrak F}$ is graded with respect to the $\ZZ^{m+1}$\!--grading described above. Thus, the algebra $R_{\mathfrak F}$ is also $\ZZ^{m+1}$\!--graded.

 Let us describe all graded components  of the algebra  $R_{\mathfrak{F}}.$
Denote by $U_{ij}\subseteq C$ the sum of the subspaces $V_i$ and $W_j.$ By Property (G), the dimension of the subspace $U_{ij}$ is equal to $n+k_i-k_j$ when $i\ge j.$
Let us choose subspaces $T_{ij}\subseteq C$ that are complements to $U_{ij}$ in $C.$ In the case when $i\ge j,$ the dimension of the subspace $T_{ij}$ is equal to $k_j-k_i.$
In particular case $k_j=k_i,$ we obtain that $T_{ij}=0.$
Let us denote by $\theta_{ij}: C\to T_{ij}$ the natural projections with kernels equal to the vector subspaces $U_{ij}.$

By property (G), we know that $V_i\cap W_i=0$ and, hence, the subspace $W_i\subseteq$ is a complement to $V_i\subseteq C$  and vice versa.
For unifying notations the vector subspaces $V_i$ and $W_j$ will be denoted by $T_{\bullet i}$ and  $T_{j\bullet},$ respectively. By $\theta_{\bullet i}:C\to T_{\bullet i}$ and $\theta_{j\bullet}:C\to T_{j\bullet}$
we also denote the natural projections with kernels isomorphic to $W_i$ and $V_j,$ respectively.

In the cases when $P=\emptyset$ or $u=0,$ the graded components $(R_{\mathfrak{F}})_{(u;P)}$ are isomorphic to $\Anm_{(u;P)}.$  For the pairs $(u;P)$ with $P\ne \emptyset$ and $u\ge 1$ we consider the following maps of vector spaces:
\begin{equation}\label{addbas}
\begin{array}{l}
\phi_{(s-1;P)}:  T_{p_s p_{s-1}}\otimes\cdots\otimes T_{p_2 p_1} \lto (R_{\mathfrak{F}})_{(s-1;P)},\\
\phi_{(s;P)}: (T_{\bullet p_s}\otimes T_{p_s p_{s-1}}\otimes\cdots\otimes T_{p_2 p_1})\oplus (T_{p_s p_{s-1}}\otimes\cdots\otimes T_{p_2 p_1}\otimes T_{p_1\bullet})\lto  (R_{\mathfrak{F}})_{(s;P)},\\
\phi_{(s+1;P)}:  T_{\bullet p_s}\otimes T_{p_s p_{s-1}}\otimes\cdots\otimes T_{p_2 p_1}\otimes T_{p_1\bullet} \lto (R_{\mathfrak{F}})_{(s+1;P)},
\end{array}
\end{equation}
which are defined by the rule:
\[
\begin{array}{l}
\phi_{(s-1;P)}(t_{p_s p_{s-1}}\otimes\cdots\otimes t_{p_2 p_1})=b_{p_s}t_{p_s p_{s-1}}b_{p_{s-1}}\cdots b_{p_2} t_{p_2 p_1}b_{p_1}, \;\;\text{where}\quad t_{ij}\in T_{ij},\\
\phi_{(s;P)}(t_{p_s}\otimes t_{p_s p_{s-1}}\otimes\cdots\otimes t_{p_2 p_1})=   t_{p_s}b_{p_s}t_{p_s p_{s-1}}b_{p_{s-1}}\cdots b_{p_2} t_{p_2 p_1}b_{p_1}, \;\;\text{where}\;\; t_{ij}\in T_{ij},\, t_{p_s}\in T_{\bullet p_s},\\
\phi_{(s;P)}(t_{p_s p_{s-1}}\otimes\cdots\otimes t_{p_2 p_1}\otimes t_{p_1})=b_{p_s}t_{p_s p_{s-1}}b_{p_{s-1}}\cdots b_{p_2} t_{p_2 p_1}b_{p_1} t_{p_1}, \;\;\text{where}\;\; t_{ij}\in T_{ij},\, t_{p_1}\in T_{p_1\bullet}, \\
\phi_{(s+1;P)}(t_{p_s}\otimes t_{p_s p_{s-1}}\otimes\cdots\otimes t_{p_2 p_1}\otimes t_{p_1} )=t_{p_s} b_{p_s}\cdots  t_{p_2 p_1} b_{p_1} t_{p_1}, \;\text{where}
\;\, t_{ij}\in T_{ij},\, t_{p_s}\in T_{\bullet p_s},\, t_{p_1}\in T_{p_1\bullet}.
\end{array}
\]
\smallskip

The following proposition gives us a description of all graded components of the algebra $R_{\mathfrak{F}}.$
\begin{proposition}\label{descr}
For any nonempty subset $P\subseteq\{1,\dots, m\}$  and any integer $u\ge 1$
the homomorphism $\phi_{(u;P)}$ is an isomorphism.
\end{proposition}
\begin{proof}
Let us consider the ideal $\overline{I}_{\mathfrak F}\subset \Anm.$  It is $\ZZ^{m+1}$\!--graded, because relations 2) and 3) of Definition \ref{defalg} are homogenious.
Let $
(\overline{I}_{\mathfrak F})_{(u; P)}\subseteq \Anm_{(u;P)}
$
be the graded component of the ideal $\overline{I}_{\mathfrak F}.$ For any $P\subseteq\{1,\dots,m\}$ the vector subspace $(\overline{I}_{\mathfrak F})_{(s-1; P)}\subseteq b_{p_s}C b_{p_{s-1}}\cdots b_{p_2} C b_{p_1}$ is generated by the subspaces of the form $b_{p_s}C\cdots b_{p_i} V_{p_i}\cdots  C b_{p_1}$ with $1< i\le s$
and by the subspaces of the form $b_{p_s}C\cdots W_{p_j}b_{p_j}\cdots C b_{p_1}$ with $1\le j<s.$ Thus, the subspace $(\overline{I}_{\mathfrak F})_{(s-1; P)}$ can be obtained as the sum of the subspaces of the form
$b_{p_s}C\cdots b_{p_i} U_{p_i p_{i-1}}b_{p_{i-1}}\cdots C b_{p_1}$ with $1< i\le s,$ where $U_{ij}$ is the sum of $V_{i}$ and $W_{j}$ as above. Now, since the subspaces $T_{ij}\subseteq C$ are complements to $U_{ij}$ in $C,$
 the subspace $b_{p_s}T_{p_s p_{s-1}}b_{p_{s-1}}\cdots b_{p_2} T_{p_2 p_1}b_{p_1}$ forms a complement to the subspace
$(\overline{I}_{\mathfrak F})_{(s-1; P)}$ in the space  $b_{p_s}C b_{p_{s-1}}\cdots b_{p_2} C b_{p_1}$ for any $P\subset\{1,\dots,m\}.$
This implies that  the maps $\phi_{(s-1; P)}$ are isomorphisms for any $P\ne \emptyset$ and $s>1.$ In a similar way, we show that maps $\phi_{(s; P)}$ and  $\phi_{(s+1; P)}$ are isomorphisms too for any
$P\ne \emptyset.$
\end{proof}

Now we describe action of  the multiplication operation on the grading components of the algebra $R_{\mathfrak{F}}.$
Consider two graded components $(R_{\mathfrak{F}})_{(u;P)}$ and $(R_{\mathfrak{F}})_{(v;Q)},$ where $P=\{p_1,\dots, p_s\}$ and $Q=\{q_1,\dots, q_t\}$ are nonempty subsets
of $\{1,\dots, m\}.$ If $p_1\le q_t,$ the multiplication $(R_{\mathfrak{F}})_{(u;P)}\cdot (R_{\mathfrak{F}})_{(v;Q)}$ is zero for any two elements from these graded components.
A nontrivial multiplication appears in the case when $p_1>q_t.$ In this situation the multiplication is defined as a natural surjective map
\[
\mu^{u,v}_{P,Q}: (R_{\mathfrak{F}})_{(u;P)}\otimes (R_{\mathfrak{F}})_{(v;Q)}\lto (R_{\mathfrak{F}})_{(v+u;Q\bigsqcup P)},
\]
which is induced by the surjective map $T_{p_1\bullet}\to C\stackrel{\theta_{p_1,q_t}}{\lto} T_{p_1 q_t},$ in the case $u\ge s, v\le t,$ and
by the surjective map $T_{\bullet q_t}\to C\stackrel{\theta_{p_1,q_t}}{\lto} T_{p_1 q_t},$  in the case $v\ge t, u\le s.$
Note that the graded components $(R_{\mathfrak{F}})_{(v+u;Q\bigsqcup P)}$ are trivial if $|v+u-t-s|>1.$
If the subset $P$ is empty, the multiplication $(R_{\mathfrak{F}})_{(1;\emptyset)}=C$ with any $(R_{\mathfrak{F}})_{(v;Q)},$ where $v\le t,$ is induced by the surjective projection
$\theta_{\bullet q_t}:C\to T_{\bullet q_t}.$ Similarly, if $Q=\emptyset,$ the multiplication  $(R_{\mathfrak{F}})_{(u;P)}, u\le s$ with $(R_{\mathfrak{F}})_{(1;\emptyset)}=C$ is induced by the surjective projection
$\theta_{p_1 \bullet}:C\to T_{p_1 \bullet}.$

With any algebra $R_{\mathfrak{F}}$ we can associate many different DG algebras that are related to different $\ZZ$\!-gradings on it.
\begin{definition}
Let $\chi:\ZZ^{m+1}\to \ZZ$ be a homomorphism of groups. By $\dR_{\mathfrak{F}}(\chi)$ we denote the DG algebra, which is the algebra $R_{\mathfrak{F}}$ equipped with the $\ZZ$\!--grading induced by the homomorphism $\chi$ and with the zero differential.
\end{definition}

Thus, the DG algebra $\dR_{\mathfrak{F}}(\chi)$ is the algebra $R_{\mathfrak{F}}$ for which the elements $c\in C$ have degree $\chi(\varepsilon_0),$ while the elements $b_i$ have degree
$\chi(\varepsilon_i)$ for each $i=1,\dots, m.$
When the homomorphism $\chi$ is trivial, we obtain our usual algebra  $R_{\mathfrak{F}}.$

\subsection{Algebras and DG algebras with two simple modules as a twisted tensor products}

In this section we describe the algebras $R_{\mathfrak{F}}$ and the DG algebras $\dR_{\mathfrak{F}}(\chi)$ as  iterated twisted tensor products of certain  algebras (and DG algebras)  that have a simple structure.
As a consequence we will be able to show that the algebras $R_{\mathfrak{F}}$ have finite global dimensions and, moreover, all DG algebras $\dR_{\mathfrak{F}}(\chi)$ are smooth.

Let $\mathfrak{F}=\{V_1, W_1;\ldots; V_m, W_m\}$ be a family of subspaces
satisfying  Property (G). Consider the family $\mathfrak{G}=\{V_1, W_1;\dots; V_{m-1}, W_{m-1}\}$ which consists of the subspaces $V_i,\; W_i$ with $i< m.$  Let $R_{\mathfrak{G}}$ be the algebra that is related to
the family $\mathfrak{G}$ in the sense of our Definition \ref{defalg}.
The algebra $R_{\mathfrak{G}}$ can be considered as a subalgebra of $R_{\mathfrak{F}}.$ Moreover, the subalgebra $R_{\mathfrak{G}}\subset R_{\mathfrak{F}}$ consists exactly those graded components
$(R_{\mathfrak{F}})_{(u;P)}$ for which $P=\{p_1,\dots, p_s\}$ is a subset of $\{1,\dots, m-1\}.$

Let us take the vector subspace $V_{m}\subseteq C$ and consider the Kronecker algebra $K(V_{m}).$
We have a natural embedding of the algebras $K(V_{m})\subset R_{\mathfrak{G}}.$ Thus,  the algebra $R_{\mathfrak{G}}$ can be considered as a $K(V_{m})$\!--ring.
The next lemma, which we will need in the sequel, is almost evident.

\begin{lemma}\label{inject}
The canonical map $V_m\otimes e_1 R_{\mathfrak{G}}\to e_2 R_{\mathfrak{G}}$ is injective.
\end{lemma}
\begin{proof}
Taking in account isomorphisms (\ref{addbas}), injectivity of the map $V_m\otimes e_1 R_{\mathfrak{G}}\to e_2 R_{\mathfrak{G}}$ follows from injectivity of the composition maps
$V_m\to C\stackrel{\theta_{\bullet j}}{\lto} T_{\bullet j}\cong V_j$ for each $j<m.$ The latter follows from property (G) which, in particular, requires that $V_m\cap W_j=0$ for any $j\le m.$
\end{proof}

Now we can show that the algebra $R_{\mathfrak{G}}$ is  a projective as the left $K(V_{m})$\!--module.

\begin{proposition}\label{RleftK} The algebra $R_{\mathfrak{G}}$ is  a projective left $K(V_{m})$\!--module.
\end{proposition}
\begin{proof}
The indecomposable left projective $K(V_{m})$\!--modules are $Q_1=K(V_{m})e_1=V_{m}\oplus\langle e_1 \rangle$ and $Q_2=K(V_{m})e_2=\langle e_2\rangle.$
Let us take $\Hom_{K(V_m)}(Q_1, R_{\mathfrak{G}})$ in the category of left  $K(V_{m})$\!--modules.
This vector space  is isomorphic to $e_1 R_{\mathfrak{G}}.$ Consider the canonical map
\[
\ev: Q_1\otimes\Hom_{K(V_m)}(Q_1, R_{\mathfrak{G}})\lto R_{\mathfrak{G}}.
\]
Since the quotient of this map is a projective module of the form $Q_2^{\oplus N},$ we have to show that the map
$\ev$ is injective. This is equivalent to check that the canonical map $V_m\otimes e_1 R_{\mathfrak{G}}\to e_2 R_{\mathfrak{G}}$ is injective.
Thus, Lemma \ref{inject} implies the proposition.
\end{proof}

The algebra $R_{\mathfrak{G}}$ admits many different augmentations $R_{\mathfrak{G}}\to K(V_{m}).$ They are directly related to specifying a complement to $V_m$ in $C.$
Now, we fix the augmentation which is related to the vector subspace $W_m\subseteq C.$ Such augmentation $\pi: R_{\mathfrak{G}}\to K(V_{m})$ is given by setting $\pi(W_{m})=0.$

Denote by $K(V_m; b_m)$  the subalgebra of $R_{\mathfrak F}$ that is generated by the subspace $V_m\subseteq C$  and the
element $b_m.$  Relation 2) of Definition \ref{defalg} says that $b_mV_m=0.$ Taking in account this relation we
obtain that the algebra $K(V_m; b_m)$ has the following nontrivial graded components:
\[
\langle e_1\rangle\oplus\langle e_2\rangle,\quad V_m\subseteq C=(R_{\mathfrak F})_{(1,\emptyset)},\quad \langle b_m\rangle=(R_{\mathfrak F})_{(0,\{m\})},\quad  V_m b_m\subseteq (R_{\mathfrak F})_{(1;\{m\})}.
\]
Moreover, the subspace $V_mb_m$ of the space $(R_{\mathfrak F})_{(1;\{m\})}=T_{\bullet m}b_m\oplus b_mT_{m\bullet}$
 coincides with the direct
summand $T_{\bullet m}b_m,$ because $V_m=T_{\bullet m}.$
It is easy to see that the algebra $K(V_m; b_m)$ is also isomorphic to the twisted tensor product of the
Kronecker algebra $K(V_{m})$  with the opposite Kronecker algebra $K_{1}^{\op}\cong\kk Q_{0,1}$ over the semisimple part
$S\cong\kk\times \kk$ via the twisting map $\mathbf{v}$ defined by formula (\ref{vtwist}). In other words, we have an isomorphism
\[
K(V_m; b_m)\cong K(V_{m})\otimes^{\mathbf{v}}_S K_{1}^{\op}.
\]
In addition, it should be noted that the algebra $K(V_m; b_m)$ is  the $K(V_{m})$\!--ring and has a canonical augmentation $K(V_{m})\otimes^{\mathbf{v}}_S K_{1}^{\op}\to K(V_{m}).$

\begin{theorem}\label{TwProd}
The algebra $R_{\mathfrak{F}}$ is isomorphic to a twisted tensor product of the algebras $(K(V_{m})\otimes^{\mathbf{v}}_S K_{1}^{\op})$ and $R_{\mathfrak{G}}$ over the algebra $K(V_{m})$
via the twisting map $\mathbf{v}$ defined by formula (\ref{vtwist}), i.e. we have
\[
R_{\mathfrak{F}}\cong (K(V_{m})\otimes^{\mathbf{v}}_S K_{1}^{\op})\otimes_{K(V_{m})}^{\mathbf{v}} R_{\mathfrak{G}}.
\]
\end{theorem}
\begin{proof}
Consider the natural map $\rho: K(V_m; b_m)\otimes_{K(V_m)} R_{\mathfrak{G}}\to R_{\mathfrak{F}}.$ Let us show that this map is an
isomorphism. As the right $K(V_m)$\!--module $K(V_m; b_m)$  has the following decomposition $K(V_m)\oplus (V_m b_m\oplus \langle b_m\rangle)$
 and the direct summand $(V_m b_m\oplus \langle b_m\rangle)$ is isomorphic to $\langle e_2\rangle^{k_m+1},$
where $\langle e_2\rangle$ is the simple right $K(V_m)$\!--module.
Thus, the tensor product $K(V_m; b_m)\otimes_{K(V_m)} R_{\mathfrak{G}}$ can be decomposed
as a direct sum of $R_{\mathfrak{G}}$ and the vector space $(V_m b_m\oplus \langle b_m\rangle)\otimes_{K(V_m)} R_{\mathfrak{G}}.$

We already know that the subalgebra $R_{\mathfrak{G}}\subset R_{\mathfrak{F}}$ consists exactly those graded components
$(R_{\mathfrak{F}})_{(u,P)}$
for which $P=\{p_1,\dots, p_s\}$ is a subset of $\{1,\dots, m-1\}.$ Now let us describe the restriction of the
map $\rho$ on the tensor product $(V_m b_m\oplus \langle b_m\rangle)\otimes_{K(V_m)} R_{\mathfrak{G}}.$

Consider the indecomposable right projective $K(V_{m})$\!--modules $Q'_1=e_1K(V_{m})=\langle e_1 \rangle$ and $Q'_2=e_2 K(V_{m})=V_m\oplus \langle e_2\rangle.$
The simple right $K(V_m)$\!--module  $\langle e_2\rangle$ can be obtained as the cokernel of the natural injective map $V_m\otimes Q'_1\to Q'_2.$
Therefore, the tensor product $\langle e_2\rangle\otimes_{K(V_m)} R_{\mathfrak{G}}$ is isomorphic to the cokernel of the map
\[
V_m\otimes Q'_1\otimes_{K(V_m)} R_{\mathfrak{G}} \lto Q'_2\otimes_{K(V_m)} R_{\mathfrak{G}}.
\]

Thus, the tensor product $\langle e_2\rangle\otimes_{K(V_m)}  R_{\mathfrak{G}}$ is isomorphic to the cokernel of the map
\[
V_m\otimes e_1 R_{\mathfrak{G}}\to e_2 R_{\mathfrak{G}}.
\]

The vector space $e_2 R_{\mathfrak{G}}\subset R_{\mathfrak{F}}$ has the following decomposition
\[
\langle e_2\rangle\oplus C\oplus\bigoplus_{
\begin{subarray}{l}
P=\{p_1,\dots,p_s\}\\
P\subseteq\{1,\dots,m-1\}
\end{subarray}
}
T_{\bullet p_s}b_{p_s}\cdots T_{p_2 p_1}b_{p_1}\oplus
\bigoplus_{
\begin{subarray}{l}
P=\{p_1,\dots,p_s\}\\
P\subseteq\{1,\dots,m-1\}
\end{subarray}
}
T_{\bullet p_s}b_{p_s}\cdots T_{p_2 p_1}b_{p_1} T_{p_1\bullet}.
\]
We know that the canonical maps $V_m\to C\stackrel{\theta_{\bullet i}}{\lto} T_{\bullet i}\cong V_i$ are injective for each $i<m,$
and the quotients $T_{\bullet i}/V_m$ are isomorphic to $T_{mi}$ for any $i<m.$ Taking in account these isomorphisms and the isomorphisms $C/V_m\cong T_{m\bullet}=W_m,$
 we obtain that the restriction of the map $\rho$
on the tensor product $\langle b_m\rangle\otimes_{K(V_m)}  R_{\mathfrak{G}}$ is injective, and under this embedding the space $\langle b_m\rangle\otimes_{K(V_m)}  R_{\mathfrak{G}}$
is equal to
\[
\langle b_m\rangle\oplus b_m T_{m\bullet}\oplus\bigoplus_{
\begin{subarray}{l}
P=\{p_1,\dots,p_s\}\\
P\subseteq\{1,\dots,m-1\}
\end{subarray}
}
b_mT_{m p_s}b_{p_s}\cdots T_{p_2 p_1}b_{p_1}\oplus
\bigoplus_{
\begin{subarray}{l}
P=\{p_1,\dots,p_s\}\\
P\subseteq\{1,\dots,m-1\}
\end{subarray}
}
b_mT_{m p_s}b_{p_s}\cdots T_{p_2 p_1}b_{p_1} T_{p_1\bullet}.
\]
Similarly, since $V_m=T_{\bullet m}$ the tensor product $V_m b_m\otimes_{K(V_m)}  R_{\mathfrak{G}}$ is embedded to $R_{\mathfrak{F}}$ and under this
embedding the space $V_m b_m\otimes_{K(V_m)}  R_{\mathfrak{G}}$  is isomorphic to
\[
T_{\bullet m} b_m\oplus T_{\bullet m} b_m T_{m\bullet}\oplus\bigoplus_{
\begin{subarray}{l}
P=\{p_1,\dots,p_s\}\\
P\subseteq\{1,\dots,m-1\}
\end{subarray}
}
T_{\bullet m} b_mT_{m p_s}b_{p_s}\cdots T_{p_2 p_1}b_{p_1}\oplus
\bigoplus_{
\begin{subarray}{l}
P=\{p_1,\dots,p_s\}\\
P\subseteq\{1,\dots,m-1\}
\end{subarray}
}
T_{\bullet m} b_mT_{m p_s}b_{p_s}\cdots T_{p_2 p_1}b_{p_1} T_{p_1\bullet}.
\]

Thus, the map $\rho$ induces an embedding of the tensor product $(V_m b_m\oplus \langle b_m\rangle)\otimes_{K(V_m)} R_{\mathfrak{G}}$ to $R_{\mathfrak{F}}$ and the image of this embedding
consists of all components $(R_{\mathfrak{F}})_{(u,P)}$ for which $m\in P.$ Therefore, the total map $\rho: K(V_m; b_m)\otimes_{K(V_m)} R_{\mathfrak{G}}\to R_{\mathfrak{F}}$ is an isomorphism. This means that the algebra
$R_{\mathfrak{F}}$ is isomorphic to a twisted tensor product of the algebras $K(V_m; b_m)$ and $ R_{\mathfrak{G}}$ over $K(V_m).$

Finally, it remains to check that the twisting map is exactly the one given by formula (\ref{vtwist}). To see this
we have to show that the product of any two elements from $\Ker\pi$ and $(V_m b_m \oplus \langle b_m\rangle)$ is zero. If
$a\in (R_{\mathfrak{G}})_{(u,P)}$ and $P\ne\emptyset,$ we have $ab_m=0$ and $aV_m b_m=0.$ If now $a\in (R_{\mathfrak{G}})_{(1,\emptyset)}=C,$ the product
$ab_m$ is equal to $0$ exactly when $a$ belongs to $W_m\subset\Ker\pi.$ Therefore, the algebra $R_{\mathfrak{F}}$ is isomorphic to the twisted
tensor product $(K(V_{m})\otimes^{\mathbf{v}}_S K_{1}^{\op})\otimes_{K(V_{m})}^{\mathbf{v}} R_{\mathfrak{G}}$ via the twisting map $\mathbf{v}$ defined by formula (\ref{vtwist}).
\end{proof}

This theorem can be easily generalized to the case of DG algebras. Indeed, let us consider a DG algebra $\dR_{\mathfrak{F}}(\chi)$ for some homomorphism $\chi:\ZZ^{m+1}\to \ZZ.$
The homomorphism $\chi$ induces a homomorphism $\chi':\ZZ^{m}\to \ZZ,$ where $\ZZ^{m}$ is the subgroup of $\ZZ^{m+1}$ generated by $\varepsilon_0,\dots, \varepsilon_{m-1}.$
We can consider the DG algebra $\dR_{\mathfrak{G}}(\chi')$ that is the algebra $R_{\mathfrak{G}}$ with $\ZZ$\!--grading induced by $\chi'$ and with zero differential.
The homomorphism $\chi$ also induces a grading on the subalgebra $K(V_m; b_m)\subset R_{\mathfrak F}.$ We denote by $K(V_{m})(\chi(\varepsilon_0))$ and $K_{1}^{\op}(\chi(\varepsilon_m))$
 the DG algebras  for which $V_m$ and $b_m$ have degree $\chi(\varepsilon_0)$ and $\chi(\varepsilon_m),$ respectively. In such case they both can be considered as DG subalgebras of the DG algebra $\dR_{\mathfrak{F}}(\chi).$
Now Theorem \ref{TwProd} directly implies the following corollary.

\begin{corollary}
The DG algebra $\dR_{\mathfrak{F}}(\chi)$ is isomorphic to the twisted tensor product of the DG algebras $K(V_{m})(\chi(\varepsilon_0))\otimes^{\mathbf{v}}_S K_{1}^{\op}(\chi(\varepsilon_m))$ and
$\dR_{\mathfrak{G}}(\chi')$ over the DG algebra $K(V_{m})(\chi(\varepsilon_0))$
via the twisting map $\mathbf{v}$ defined by formula (\ref{vtwist}), i.e. we have
\[
\dR_{\mathfrak{F}}(\chi)\cong \left( K(V_{m})(\chi(\varepsilon_0))\otimes^{\mathbf{v}}_S K_{1}^{\op}(\chi(\varepsilon_m)) \right)\otimes_{K(V_{m})(\chi(\varepsilon_0))}^{\mathbf{v}} \dR_{\mathfrak{G}}(\chi').
\]
\end{corollary}

Applying Theorem \ref{DGfinpr} to the algebras $R_{\mathfrak{F}}$ and to the DG algebras $\dR_{\mathfrak{F}}(\chi),$ we obtain the following statement.

\begin{theorem}\label{smoothtwist}
For any family ${\mathfrak F}$ which satisfies property (G), the algebra $R_{\mathfrak{F}}$ has finite global dimension and  any DG algebra $\dR_{\mathfrak{F}}(\chi)$ is smooth.
\end{theorem}

\section{Some particular cases and known examples}

\subsection{Green's algebras}
In \cite{G} E.L.~Green gave an example of a series of
finite-dimensional algebras, which we denote by $G_l, l\ge 0,$ having precisely two isomorphism classes of simple
modules and being of global dimension $l$ (see also \cite{Ha}). In particular, this example showed that a values of the global
dimension cannot only depend on the number of simple modules. Besides, it was proved in \cite{Ha, MH} that the triangulated categories $\prf G_l\cong D^{b}(\mod-G_{l})$ do not have full exceptional collections in the case when $l\ge 3.$ Now we show that the algebras $G_l$ provide a special case of algebras of type $R_{\mathfrak F}$ described in the previous sections.

Let us recall a definition of the algebras $G_l$ in terms of quiver algebras.
Take the quivers with two vertices  $Q_{n+1,n}$ and $Q_{n,n}$  as in (\ref{TwoVQ}) and
define algebras $G_{2n+1}$ and $G_{2n}$ as the quotient algebras $\kk Q_{n+1,n}/I_{2n+1}$ and $\kk Q_{n,n}/I_{2n},$
where the ideals  $I_{2n+1}$ and $I_{2n}$ are generated by the following elements:
\begin{flalign}\label{Green}
\begin{array}{rll}
\mathrm{i)} & b_ic_j & \text{when}\quad i<j;\\
\mathrm{ii)} & c_jb_i & \text{when}\quad j\le i.\\
\end{array}&&
\end{flalign}
Note that for $l=0,1$ we have $G_0=S=\kk\times\kk$ and $G_1=K_1=\kk Q_{1,0}.$

At first, it can be shown that any algebra $G_l$ can be realized as an iterated twisted tensor product of  algebras of type $K_1.$
More precisely, it is easy to see that
for any $n\ge 1$ there are the following isomorphisms of algebras:
\[
G_{2n+1}\cong K_1 \otimes^{\mathbf{v}}_S G_{2n} \quad\text{and}\quad G_{2n}\cong K_1^{\op}\otimes^{\mathbf{v}}_S G_{2n-1},
\]
where as above $K_1=\kk Q_{1,0}$ and $K_1^{\op}=\kk Q_{0,1}.$
A proof goes by induction, where the first factors $K_1$ and $K_1^{\op}$ are related to the arrows $c_{n+1}$ and $b_n,$ respectively.

On the other hand, for any $l$ each algebra $G_{l}$ can be described as an algebra of the form $R_{\mathfrak F}$ for some special family $\mathfrak{F}.$
Indeed, let $n=[l/2].$
The dimension of the vector space $C$ is equal to $n,$ when $l=2n,$ or it is equal to $n+1,$ when $l=2n+1.$ In any case we have that the dimension of
$C$ is equal to $[(l+1)/2].$

Now we put $W_i=\langle c_1,\dots, c_i\rangle\subseteq C$ and $V_i=\langle c_{i+1},\dots \rangle\subset C$ for any $i=1,\dots, n.$
It is easy to see that relation i) of (\ref{Green}) coincides with relation 2) of Definition \ref{defalg}, while
relation ii) of (\ref{Green})  is exactly relation 3) of Definition \ref{defalg}.
We also need to check that 1) of Definition \ref{defalg} holds.
It is evident that  $V_i+W_j=C,$ when $i\le j.$ Hence,
$b_icb_j=0$ in $G_l$ for any $c\in C,$ when $i\le j.$ Thus, the algebra $G_l$ is isomorphic to $R_{\mathfrak F}$ for the family
$\mathfrak{F}=\{V_1, W_1;\dots; V_n,W_n\},$ where $n=[l/2]$ and the subspaces $V_i,\; W_i$ as above. Finally, it can be noted that property (G) also holds.

\begin{proposition}
Any Green algebra $G_l$ is an algebra of type $R_{\mathfrak F}$ for a family ${\mathfrak F}$ described above.
\end{proposition}

We can also consider a generalization of Green algebras to the case of $N$ simple modules.
We fix a natural number $N\ge 2$ and consider the semisimple algebra $S=S_N=\underbrace{\kk\times\dots\times\kk}_N.$

Denote by $K_{ij}[d]$ for $1\le i\ne j\le N$  the finite-dimensional DG algebras with the semisimple part equal to $S=S_N$ and with only one arrow  from vertex
$i$ to $j$ of degree $d$ in the Jacobson radical.
Taking iterated twisted tensor product over $S=S_N$  of such DG algebras  with the twisting map given by formula (\ref{vtwist}), we obtain
new DG algebras, which will be called {\sf generalized Green DG algebras on $N$ vertices}.
By Definition \ref{TwistDGal} of the twisted tensor product of DG algebras,  any such generalized Green DG algebra has a trivial differential.
The next proposition  directly follows from Theorem \ref{DGfinpr}.

\begin{proposition}\label{GreenDG}
Any generalized Green DG algebra is smooth.
\end{proposition}

\subsection{Kirkman--Kuzmanovich algebras}

Another example of series of algebras with two simple modules is due to E. Kirkman and J. Kuzmanovich (see \cite{KK}). They constructed  finite-dimensional algebras, which will be  denoted by  $KK_n,\; n\ge 1,$ with $(\rad KK_n)^4=0$ for any $n.$
In \cite{KK} it was shown that the global dimension of these algebras is equal to $2n+1,$ when $n\ge 2.$ Thus, this example demonstrates that a bound on the global
dimension of a finite-dimensional algebra cannot depend only on the Loewy length of this algebra.

The algebras $KK_n, \; n\ge 1$ are defined as the quotient algebras $\kk Q_{n,n}/J_{n},$
where the ideals  $J_n$ are generated by the following list of elements:
\begin{flalign}\label{KirKuz}
\begin{array}{rlll}
\mathrm{i)} & b_ic b_j & \text{for any}& c\in C\quad \text{and all}\quad 1\le i,j\le n;\\
\mathrm{ii)} & b_ic_i & \text{for any} &1\le i\le n;\\
\mathrm{iii)} & c_jb_i & \text{when} & j< i;\\
\mathrm{iv)} & (c_jb_i-c_i b_i) & \text{when}& j> i.
\end{array}&&
\end{flalign}

The algebras $KK_n$ can also be presented as algebras of the form $R_{\mathfrak{F}}$ for some families ${\mathfrak F}=\{V_1, W_1,\dots, V_n,W_n\}.$
We put $V_i=\langle c_{i}\rangle\subset C$ and  $W_i=\langle c_1,\dots,c_{i-1},c_{i+1}-c_i,\dots, c_n-c_i\rangle\subset C$
for any $i=1,\dots, n.$ Thus, in this case, we have $m=n$ and all $k_i$ are equal to $1.$

It is evident that relation ii) of (\ref{KirKuz}) is equivalent to relation 2) of Definition \ref{defalg}.
Similarly, relations iii) and iv) together are equivalent to 3) of Definition \ref{defalg} with $W_i=\langle c_1,\dots,c_{i-1},c_{i+1}-c_i,\dots, c_n-c_i\rangle$
introduced above.
We can see that property (G) holds for the vector subspaces $V_i,\; W_j,$ i.e. $V_i\cap W_j=0$ when $i\ge j.$
Finally, property i) of (\ref{KirKuz})  implies relation 1) of Definition \ref{defalg}.
In the opposite direction, relation 1) of Definition \ref{defalg} implies property i) of (\ref{KirKuz}) for $i\le j.$ In the case $i>j$ relation i) of (\ref{KirKuz})  directly follows from
ii)--iv), because $V_i+ W_j=C$ when $i\ge  j.$

Thus, we obtain the following proposition.

\begin{proposition}
Any Kirkman--Kuzmanovich algebra $KK_n$ is an algebra of type $R_{\mathfrak F}$ for a family ${\mathfrak F}$ described above.
\end{proposition}

\section{Equidimensional families and embeddings to categories with exceptional collection}\label{EqdimF}

\subsection{Equidimensional families and graded quivers}
In this section we  consider in more detail DG algebras $\dR_{\mathfrak F}(\chi)$ that are related to families ${\mathfrak F}$ for which all vector subspaces $V_i\subseteq C, \; i=1,\dots, m$ have the same dimension equal to $k.$
In this case we say that the family $\mathfrak{F}$ is equidimensional.

Recall that a graded quiver is a usual quiver  $Q=(\vrt, \Arr, s, t)$ with a given $\ZZ$\!--grading on the set of all arrows. Thus,  we add some map $\Delta:\Arr\to\ZZ$ to a usual quiver $Q.$
A graded quiver with relation $(Q,I,\Delta)$ is a usual quiver $Q$ with $\ZZ$\!--grading $\Delta$ such that the ideal of relations $I$ is a graded ideal.
In this case the path algebra $\kk Q/I$ is also  $\ZZ$\!--graded.

Let $\overline{\delta}=\{\delta_1,\dots, \delta_m\}$ be a sequence of integers. Let us for each vector subspace $V_i\subseteq C$ choose a basis $\{v_{i1},\dots, v_{ik}\}.$
We introduce in consideration a graded quiver with relations $(\Phi_{\mathfrak{F}}(\overline{\delta}), J_{\mathfrak F}),$ where
the graded quiver $\Phi_{\mathfrak{F}}(\overline{\delta})$ has $m+2$ vertices $\bone,\btwo, {\mathbf l_1},\dots, {\mathbf l_m}$ and the following set of arrows:
\begin{flalign*}
\begin{array}{rl}
\mathrm{a1)} & \text{$n$ arrows $\{c_1,\dots, c_n\}$ of degree $0$ from $\bone$ to $\btwo,$} \\
\mathrm{a2)} & \text{one arrow $\beta_i$ of degree $0$ from $\btwo$ to ${\mathbf l_i}$ for each $i=1,\dots, m,$}\\
\mathrm{a3)} & \text{$k$ arrows $\phi_{i1},\dots, \phi_{ik}$ of degree $\delta_i$ from $\btwo$ to ${\mathbf l_i}$ for each $i=1,\dots, m,$}\\
\end{array}&&
\end{flalign*}
while the ideal of relations $J_{\mathfrak{F}}$ is  generated by the following elements:
\begin{flalign*}
\begin{array}{rlll}
\mathrm{r1)} & \beta_i v & \text{for all}\quad v\in V_i\subseteq C,& \text{where}\quad i=1,\dots, m,\\
\mathrm{r2)} & \phi_{ij} w & \text{for all}\quad w\in W_i\subseteq C,& \text{where}\quad i=1,\dots, m,\quad 1\le j\le k,\\
\mathrm{r3)} & \phi_{ij} v_{il} & \text{for all}\quad v_{il}\in V_i\;\; \text{such that}\; l\ne j,& \text{where}\quad i=1,\dots, m,\quad 1\le j\le k,\\
\mathrm{r4)} & \phi_{ij} v_{ij}-\phi_{i1}v_{i1} & \text{for all}\quad v_{ij}\in V_i, & \text{where}\quad i=1,\dots, m,\quad 1< j\le k.\\
\end{array}&&
\end{flalign*}
\smallskip

Relations r2)--r4) imply that the vector space generated by arrows $\{\phi_{i1},\dots, \phi_{ik}\}$ is canonically isomorphic to the dual vector space $V_i^{\vee}$ for every $i,$ and, moreover, the subspace of relations $Z_i$ in the tensor product $V_i^{\vee}\otimes C$ is a direct sum of $V_i\otimes W_i$ and the subspace $\Ad(V_i)\subset V_{i}^{\vee}\otimes V_i\cong \End(V_i)$ of codimension one consisting of all traceless tensors.

Let us consider the DG algebra of the DG quiver with relations $(\Phi_{\mathfrak{F}}(\overline{\delta}), J_{\mathfrak F}),$ i.e.  the quotient DG algebra $\dD_{\mathfrak{F}}(\overline{\delta})=\kk \Phi_{\mathfrak{F}}(\overline{\delta})/J_{\mathfrak F},$ where $\kk \Phi_{\mathfrak{F}}(\overline{\delta})$ is the path DG algebra of the DG quiver $\Phi_{\mathfrak{F}}(\overline{\delta}).$
The DG algebra $\dD_{\mathfrak{F}}(\overline{\delta})$ as the right DG module is isomorphic to a direct sum of $m+2$ semi-projective DG modules $\mQ_{\bone}=e_{\bone}\dD_{\mathfrak{F}}(\overline{\delta}),$ $\mQ_{\btwo}=e_{\btwo}\dD_{\mathfrak{F}}(\overline{\delta})$ and $\mQ_{\mathbf l_i}=e_{\mathbf l_i}\dD_{\mathfrak{F}}(\overline{\delta}),\; i=1,\dots, m.$
The following proposition is a direct consequence of the definition of the DG algebra $\dD_{\mathfrak{F}}(\overline{\delta}).$

\begin{proposition}\label{excol}
The DG $\dD_{\mathfrak{F}}(\overline{\delta})$\!--modules $\mQ_{\bone}=e_{\bone}\dD_{\mathfrak{F}}(\overline{\delta}),$ $\mQ_{\btwo}=e_{\btwo}\dD_{\mathfrak{F}}(\overline{\delta})$ and $\mQ_{\mathbf l_i}=e_{\mathbf l_i}\dD_{\mathfrak{F}}(\overline{\delta}),\; i=1,\dots, m,$ form a full exceptional collection in the triangulated category $\prf\dD_{\mathfrak{F}}(\overline{\delta})$
such that the objects $\mQ_{\mathbf l_i},\; i=1\dots, m,$ are completely orthogonal.
\end{proposition}
Furthermore, the following proposition provides an explicit   description of the structure of the DG algebra $\dD_{\mathfrak{F}}(\overline{\delta}).$
\begin{proposition}\label{excolhom}
All nonvanishing complexes of homomorphisms between semi-projective DG modules $\mQ_{\bone}=e_{\bone}\dD_{\mathfrak{F}}(\overline{\delta}),\; \mQ_{\btwo}=e_{\btwo}\dD_{\mathfrak{F}}(\overline{\delta})$ and $\mQ_{\mathbf l_i}=e_{\mathbf l_i}\dD_{\mathfrak{F}}(\overline{\delta}),\; i=1,\dots, m$ in the DG category
$\Mod\dD_{\mathfrak{F}}(\overline{\delta})$ can be described as follows:
\begin{flalign*}
\begin{array}{rlll}
\mathrm{(1)} & \dHom_{\dD_{\mathfrak{F}}(\overline{\delta})}(\mQ_{a}, \mQ_{a})\cong \kk & \text{for any}\quad a=\bone, \btwo,{\mathbf l_1},\dots, {\mathbf l_m};\\
\mathrm{(2)} & \dHom_{\dD_{\mathfrak{F}}(\overline{\delta})}(\mQ_{\bone}, \mQ_{\btwo})\cong C; &\\
\mathrm{(3)} & \dHom_{\dD_{\mathfrak{F}}(\overline{\delta})}(\mQ_{\btwo}, \mQ_{\mathbf l_i})\cong \kk\oplus V_i^{\vee}[-\delta_i]& \text{for any}\quad i=1,\dots, m;\\
\mathrm{(4)} & \dHom_{\dD_{\mathfrak{F}}(\overline{\delta})}(\mQ_{\bone}, \mQ_{\mathbf l_i})\cong W_i\oplus \kk[-\delta_i] & \text{for any}\quad i=1,\dots, m.\\
\end{array}&&
\end{flalign*}
All nontrivial compositions are obtained from the canonical projections
$C\lto W_i$ and canonical maps $V_i^{\vee}\otimes C\to V_i^{\vee}\otimes V_i\to \kk$ for all $i=1,\dots, m$ (see Fig. \ref{quiver1}).
\end{proposition}
\begin{proof}
The isomorphisms (1)--(3) directly follow from  a1)--a3) of the definition of the graded quiver $\Phi_{\mathfrak{F}}.$
Condition r1) implies that $\beta_i V_i=0$ for any $i.$ Therefore, we obtain
\[
\dHom_{\dD_{\mathfrak{F}}(\overline{\delta})}^0(\mQ_{\bone}, \mQ_{\mathbf l_i})\cong C/V_i\cong W_i.
\]
Finally, relations r2)--r4) imply that the subspace  $Z_i=(V_i^{\vee}\otimes W_i)\oplus\Ad(V_i)\subset V_i^{\vee}\otimes C$ belongs to the ideal $I_{\mathfrak{F}}.$ Hence, we have
$
\dHom_{\dD_{\mathfrak{F}}(\overline{\delta})}^{\delta_i}(\mQ_{\bone}, \mQ_{\mathbf l_i})\cong (V_i^{\vee}\otimes C)/Z_i\cong \kk.
$
\end{proof}

\begin{figure}[ht]
\hspace*{0cm}
\xymatrix{
&&&& \bullet\, {\mathbf l_{1}}\\
&&&& \bullet\, {\mathbf l_{2}}\\
\underset{\bone}{\bullet}
\ar@2{->}@/^0.8pc/@<0.9ex>[urrrr]
\textcolor{red}{}
\ar@[red]@{-->}@/^0.8pc/@<0.1ex>[urrrr]
\ar@2{->}@/^1pc/@<1.4ex>[uurrrr]
\textcolor{red}{}
\ar@[red]@{-->}@/^1pc/@<0.6ex>[uurrrr]
\ar@2{->}@/_0.8pc/@<-0.8ex>[drrrr]
\textcolor{red}{}
\ar@[red]@{-->}@/_0.8pc/@<-1.6ex>[drrrr]
\ar@<0.8ex>[rr]
\ar@<0.4ex>[rr]
\ar[rr]
\ar@<-0.4ex>[rr]
\ar@<-0.8ex>[rr] &&
\underset{\btwo}{\bullet}
\ar@<0.9ex>[uurr] |(0.50)\hole|(0.57)\hole
\textcolor{red}{}
\ar@[red]@3{->}@{-->}@<0.1ex>[uurr] |(0.52)\hole|(0.60)\hole
\ar@<-0.1ex>[urr]
\textcolor{red}{}
\ar@[red]@3{->}@{-->}@<-0.9ex>[urr]
\ar@<0.2ex>[drr]
\textcolor{red}{}
\ar@[red]@3{->}@{-->}@<-0.8ex>[drr]&&
\genfrac{}{}{0pt}{}{\vdots}{\vdots}\\
&&&& \bullet\, {\mathbf l_{m}}
}
\caption{The DG algebra $\dD_{\mathfrak{F}}(\overline{\delta})$}
\label{quiver1}
\end{figure}

\subsection{Another semiorthogonal decomposition}
Let us now consider the following semi-projective DG $\dD_{\mathfrak{F}}(\overline{\delta})$\!--modules
\begin{equation}\label{P-mod}
\mP_1=\Cone\left(\mQ_{\bone}\stackrel{\can}{\lto} \bigoplus_{i=1}^{m} \mQ_{\mathbf l_i}[\delta_i]\right), \qquad \mP_2= \Cone\left(\mQ_{\btwo}\stackrel{\can}{\lto} \bigoplus_{i=1}^{m} V_i\otimes \mQ_{\mathbf l_i}[\delta_i]\right).
\end{equation}
This means that there is a filtration on the DG module $\mP_1$ with  $\bigoplus_{i=1}^{m} \mQ_{\mathbf l_i}[\delta_i]\hookrightarrow\mP_1$ as a DG submodule and $\mQ_{\bone}[1]$ as the quotient, while the differential on $\mP_1$ is induced by the canonical map $\can: \mQ_{\bone}\to \bigoplus_{i=1}^{m} \mQ_{\mathbf l_i}[\delta_i].$
Similarly, for the DG module $\mP_2$ we have a filtration with $\bigoplus_{i=1}^{m} V_i\otimes \mQ_{\mathbf l_i}[\delta_i]\hookrightarrow\mP_2$ as a submodule and
$\mQ_{\btwo}[1]$ as the quotient.

For each $i=1,\dots, m,$ we also consider the following  semi-projective DG $\dD_{\mathfrak{F}}(\overline{\delta})$\!--modules
\[
\mK_i= \Tot\left(V_i\otimes \mQ_{\bone}\stackrel{\eta_i}{\lto} \mQ_{\btwo}\stackrel{\beta_i}{\lto} \mQ_{\mathbf l_i}\right),
\]
where $\eta_i$ is the composition of the natural embedding $V_i\otimes \mQ_{\bone}\hookrightarrow C\otimes \mQ_{\bone}$ and the canonical map $\can: C\otimes \mQ_{\bone}\to \mQ_{\btwo}.$
It is easy to see that the composition $\beta_i\eta_i$ is equal to $0.$ Thus, the DG module $\mK_i$ has a filtration with $\mQ_{\mathbf l_i}\subset \mK_i$ as a submodule  and $\Cone(\eta_i)[1]$ as the quotient.
\begin{proposition}
The sequence of DG modules $\{\mK_1,\dots, \mK_m\}$ forms an exceptional collection in the triangulated category $\prf\dD_{\mathfrak{F}}(\overline{\delta}).$ The DG modules $\mP_1$ and  $\mP_2$ belong to the left orthogonal
${}^{\perp}\cS,$ where $\cS\subset\prf\dD_{\mathfrak{F}}(\overline{\delta})$ is the full triangulated subcategory generated by the objects $\mK_1,\dots, \mK_m.$
\end{proposition}

\begin{proof}
For any $1\le i,j\le m,$ let us consider the complex $\dHom_{\dD_{\mathfrak{F}}(\overline{\delta})}(\mK_i, \mK_j)$ of morphisms between objects $\mK_i$ and $\mK_j$ in the DG category $\Mod\dD_{\mathfrak{F}}(\overline{\delta})$ of all right DG $\dD_{\mathfrak{F}}(\overline{\delta})$\!--modules.
Since all DG modules $\mK_i$ are semi-projective, the cohomology of the complex $\dHom_{\dD_{\mathfrak{F}}(\overline{\delta})}(\mK_i, \mK_j)$ determine the spaces of morphisms $\Hom(\mK_i, \mK_j[l])$ between these objects in the triangulated category $\prf\dD_{\mathfrak{F}}(\overline{\delta}).$
For any $i,j,$ the complex $\dHom_{\dD_{\mathfrak{F}}(\overline{\delta})}(\mK_i, \mK_j)$ is a direct sum of two complexes; namely, the complex
\begin{multline}\label{firstmap}
V_i^{\vee}\otimes V_j\otimes \dHom_{\dD_{\mathfrak{F}}(\overline{\delta})}^0(\mQ_{\bone}, \mQ_{\bone})\oplus \dHom_{\dD_{\mathfrak{F}}(\overline{\delta})}^0(\mQ_{\btwo}, \mQ_{\btwo})\oplus \dHom_{\dD_{\mathfrak{F}}(\overline{\delta})}^{0}(\mQ_{\mathbf l_i}, \mQ_{\mathbf l_j})\lto\\
\lto V_{i}^{\vee}\otimes \dHom_{\dD_{\mathfrak{F}}(\overline{\delta})}^0(\mQ_{\bone}, \mQ_{\btwo})\oplus
\dHom_{\dD_{\mathfrak{F}}(\overline{\delta})}^0(\mQ_{\btwo}, \mQ_{\mathbf l_j})\lto
 V_i^{\vee}\otimes \dHom_{\dD_{\mathfrak{F}}(\overline{\delta})}^0(\mQ_{\bone}, \mQ_{\mathbf l_j})
\end{multline}
concentrated in degrees $0,1,2$ and the complex
\begin{equation}\label{secondmap}
\dHom_{\dD_{\mathfrak{F}}(\overline{\delta})}^{\delta_j}(\mQ_{\btwo}, \mQ_{\mathbf l_j})\lto V_i^{\vee}\otimes \dHom_{\dD_{\mathfrak{F}}(\overline{\delta})}^{\delta_j}(\mQ_{\bone}, \mQ_{\mathbf l_j}),
\end{equation}
which is concentrated in degrees $\delta_j+1$ and $\delta_j+2.$ The map (\ref{secondmap}) is actually the map $V_j^{\vee}\to V_i^{\vee},$ which is dual to the natural map
$V_i\to C\to V_j.$ The last map is an isomorphism when $V_i\cap W_j=0.$ By assumption (G) it holds for all $i\ge j.$
Thus, the map (\ref{secondmap}) is an isomorphism when $i\ge j.$

Now consider the complex of vector spaces (\ref{firstmap}). It coincides with the complex
\begin{equation}\label{firstmapnew}
(V_i^{\vee}\otimes V_j)\oplus \kk\oplus \dHom_{\dD_{\mathfrak{F}}(\overline{\delta})}^{0}(\mQ_{\mathbf l_i}, \mQ_{\mathbf l_j})\lto (V_{i}^{\vee}\otimes C)\oplus
\kk
\lto V_i^{\vee}\otimes W_j,
\end{equation}
where $V_i^{\vee}\otimes V_j\to V_{i}^{\vee}\otimes C$ is the natural injection and $V_{i}^{\vee}\otimes C\to V_i^{\vee}\otimes W_j$ is the natural projection. It is easy to see that the complex (\ref{firstmapnew}) has only one nontrivial $0$\!--cohomology, which is isomorphic to the vector space
$\dHom_{\dD_{\mathfrak{F}}(\overline{\delta})}^0(\mQ_{\mathbf l_i}, \mQ_{\mathbf l_j}).$ It is actually trivial when $i\ne j$ and it is isomorphic to $\kk$ in the case $i=j.$
Therefore, in the triangulated category $\prf\dD_{\mathfrak{F}}(\overline{\delta})$ objects $\mK_i,\; i=1,\dots, m,$  are exceptional  and form an exceptional collection $\{\mK_1,\dots, \mK_m\}.$

In a similar way we can describe the complexes of vector spaces $\dHom_{\dD_{\mathfrak{F}}(\overline{\delta})}(\mP_i, \mK_j),$ where $i=1,2$ and $j=1,\dots, m.$
Straightforward calculations show that for any $j$ the complex $\dHom_{\dD_{\mathfrak{F}}(\overline{\delta})}(\mP_2, \mK_j)$  is a direct sum of two complexes
\begin{equation}\label{P2K}
\dHom_{\dD_{\mathfrak{F}}(\overline{\delta})}^0(\mQ_{\btwo}, \mQ_{\btwo})\to \dHom_{\dD_{\mathfrak{F}}(\overline{\delta})}^0(\mQ_{\btwo}, \mQ_{\mathbf l_j})\quad \text{and}
\quad
V_{j}^{\vee}\otimes\dHom_{\dD_{\mathfrak{F}}(\overline{\delta})}^{0}(\mQ_{\mathbf l_j}, \mQ_{\mathbf l_j})\to
\dHom_{\dD_{\mathfrak{F}}(\overline{\delta})}^{\delta_j}(\mQ_{\btwo}, \mQ_{\mathbf l_j}),
\end{equation}
which are obviously both acyclic. Hence, the object $\mP_2$ belongs to ${}^{\perp}\cS\subset\prf\dD_{\mathfrak{F}}(\overline{\delta}).$

Finally, let us consider the complexes $\dHom_{\dD_{\mathfrak{F}}(\overline{\delta})}(\mP_1, \mK_j)$ for each $i=1,\dots, m.$ The complex $\dHom_{\dD_{\mathfrak{F}}(\overline{\delta})}(\mP_1, \mK_j)$ is a direct sum of two complexes
\begin{equation}
V_j\lto C\lto W_j\qquad\text{and}\qquad \kk\stackrel{\sim}{\lto}\kk
\end{equation}
concentrated in degrees $-1,0,1$ and $\delta_j, \delta_j+1,$ respectively.
Since both these complexes are acyclic, the DG module $\mP_1$ as the object of the triangulated category $\prf\dD_{\mathfrak{F}}(\overline{\delta})$ belongs to the subcategory ${}^{\perp}\cS\subset\prf\dD_{\mathfrak{F}}(\overline{\delta})$ too.
\end{proof}

\begin{proposition}\label{gencat} The set of the DG $\dD_{\mathfrak{F}}(\overline{\delta})$\!--modules $\mP_1, \mP_2, \mK_1,\dots, \mK_m$ generates the triangulated category $\prf\dD_{\mathfrak{F}}(\overline{\delta}).$
\end{proposition}
\begin{proof}
Denote by $\T\subseteq\prf\dD_{\mathfrak{F}}(\overline{\delta})$ the triangulated subcategory generated by $\mP_1, \mP_2, \mK_1,\dots, \mK_m.$
To show that the category $\T$ coincides with the whole category $\prf\dD_{\mathfrak{F}}(\overline{\delta}),$ it is sufficient to check that all DG modules
$\mQ_{\bone}, \mQ_{\btwo},$ and $\mQ_{\mathbf l_i},\;i=1,\dots, m,$ belong to $\T.$

Let us consider the embedding of vector spaces $V_m\subset C\hookrightarrow \dHom_{\dD_{\mathfrak{F}}(\overline{\delta})}^{0}(\mP_1,\mP_2).$ This embedding induces a map $V_m\otimes \mP_1\to\mP_2.$
The property (G) implies that the canonical maps $V_m\to C\stackrel{\theta_{i\bullet}}{\to} V_i$ are isomorphisms for all $i.$ Therefore, the cone of the map
$V_m\otimes \mP_1\to\mP_2$ is isomorphic to the cone $C(\eta_m)$ of the map $V_m\otimes \mQ_{\bone}[1]\stackrel{\eta_m}{\to} \mQ_{\btwo}[1]$ in the triangulated category $\prf\dD_{\mathfrak{F}}(\overline{\delta}).$
On the other hand, we have an exact sequence $\mQ_{\mathbf l_m}\to \mK_m\to C(\eta_m)$ in the category $\prf\dD_{\mathfrak{F}}(\overline{\delta}).$ Thus, the object $\mQ_{\mathbf l_m}$ belongs to the subcategory $\T.$
Now we can take the DG modules $\mP_1'$ and $\mP_2'$ which are the quotients $\mP_1/\mQ_{\mathbf l_m}[\delta_m]$ and $\mP_2/V_m\otimes\mQ_{\mathbf l_m}[\delta_m].$
Both these DG modules belongs to the subcategory $\T.$ Repeating the procedure above with $V_{m-1}\subset C$ and the DG modules $\mP_1', \mP_2', \mK_{m-1},$ we can show that the DG module
$\mQ_{\mathbf l_{m-1}}$ also belongs to $\T.$ This implies that all $\mQ_{\mathbf l_i}$ are in the subcategory $\T.$ Finally, the DG modules $\mQ_{\bone}$ and $\mQ_{\btwo}$ also belong to the subcategory $\T$
as the quotients of $\mP_1$ and $\mP_2$ by submodules generated by $\mQ_{\mathbf l_i}.$
\end{proof}

\subsection{Embedding to categories with exceptional collection}

Let us now take the semi-projective DG $\dD_{\mathfrak{F}}(\overline{\delta})$\!--module $\mP_1\oplus\mP_2$ and denote by $\dE_{\mathfrak{F}}(\overline{\delta})$ the DG algebra of endomorphisms of this DG module, i.e $\dE_{\mathfrak{F}}(\overline{\delta})=\dEnd_{\dD_{\mathfrak{F}}(\overline{\delta})}(\mP_1\oplus\mP_2).$

\begin{proposition}\label{quasi-is}
The DG algebra $\dE_{\mathfrak{F}}(\overline{\delta})$ is quasi-isomorphic to the DG algebra $\dR_{\mathfrak F}(\chi),$ where $\chi:\ZZ^{m+1}\to\ZZ$ is such a homomorphism that
$\chi(\varepsilon_0)=0$ and $\chi(\varepsilon_i)=1-\delta_i$ for every $i=1,\dots, m.$
\end{proposition}
\begin{proof}
 The DG algebra $\dR_{\mathfrak F}(\chi)$ is isomorphic to the quotient DG algebra $\kk Q_{n,m}(\chi)/I_{\mathfrak F},$ where
$Q_{n,m}(\chi)$ is the DG quiver that is the quiver (\ref{TwoVQ}) with the grading $\chi.$ Let us construct a homomorphism of the DG algebras
$\xi:\kk Q_{n,m}(\chi)\to \dE_{\mathfrak{F}}(\overline{\delta})$ and show that this homomorphism induces a homomorphism $\widebar{\xi}: \dR_{\mathfrak F}(\chi)\to \dE_{\mathfrak{F}}(\overline{\delta}),$ which is a quasi-isomorphism.
To construct the map $\xi,$ it is sufficient to define $\xi(b_i)$ for all $i=1,\dots, m$ and $\xi(c)$ for any $c\in C.$

Consider the semi-projective DG modules $\mP_1,\, \mP_2.$ We have that
$\mP_1=\bigoplus_{i=1}^{m} \mQ_{\mathbf l_i}[\delta_i]\oplus \mQ_{\bone}[1]$ and  $\mP_2= \bigoplus_{i=1}^{m} V_i\otimes \mQ_{\mathbf l_i}[\delta_i]\oplus \mQ_{\btwo}[1]$ with differential induced by
the canonical maps (\ref{P-mod}).

For any $i=1,\dots, m$ we define $\xi(b_i)\in\dHom_{\dD_{\mathfrak{F}}(\overline{\delta})}^{1-\delta_i}(\mP_2,\mP_1)$ as the image of $\beta_i\in \dHom_{\dD_{\mathfrak{F}}(\overline{\delta})}^{0}(\mQ_{\btwo},\mQ_{\mathbf l_i})$ under the natural map
$
\psi_i:\dHom_{\dD_{\mathfrak{F}}(\overline{\delta})}(\mQ_{\btwo}[1],\mQ_{\mathbf l_i}[\delta_i])\to \dHom_{\dD_{\mathfrak{F}}(\overline{\delta})}(\mP_2,\mP_1).
$
It is easy to see that the induced map
\[
\bigoplus_{i=1}^m \dHom_{\dD_{\mathfrak{F}}(\overline{\delta})}^{0}(\mQ_{\btwo},\mQ_{\mathbf l_i})[\delta_i-1]\lto \dHom_{\dD_{\mathfrak{F}}(\overline{\delta})}(\mP_2,\mP_1)
\]
is a quasi-isomorphism, because the complement to the image of this map is
the acyclic complex
$
\bigoplus_{i=1}^{m} V_i^{\vee}\stackrel{\sim}{\to} \bigoplus_{i=1}^{m} V_i^{\vee}.
$

The map $\xi(c)\in\dHom_{\dD_{\mathfrak{F}}(\overline{\delta})}^{0}(\mP_1,\mP_2)$ can be defined as the map that has  the following components
\[
c:\mQ_{\bone}[1]\to\mQ_{\btwo}[1]\qquad\text{ and}\qquad \theta_{i\bullet}(c)\otimes\id:\mQ_{\mathbf l_i}[\delta_i]\to V_i\otimes\mQ_{\mathbf l_i}[\delta_i],
\]
 where $\theta_{i\bullet}: C\to V_i=T_{i\bullet}$ is the natural projections.
It is easy to see that the map $\xi(c)$ commutes with differentials and we obtain a well-defined homomorphism of DG algebras $\xi:\kk Q_{n,m}(\chi)\to \dE_{\mathfrak{F}}(\overline{\delta}).$

Let us show that this homomorphism induces a homomorphism $\widebar{\xi}: \dR_{\mathfrak F}(\chi)\to \dE_{\mathfrak{F}}(\overline{\delta}).$ Thus, we have to check that relations 1)--3) of Definition \ref{defalg} hold.
From construction we see that the composition $\xi(b_i)\xi(c)\xi(b_j)=0$ for all $i,j.$ The relation 2) of Definition \ref{defalg} directly follows from r1) of the definition of $I_{\mathfrak F}$
which says that the composition $\beta_i v: \mQ_{\bone}\to \mQ_{\mathbf l_i}[\delta_i]$ is equal to zero, when $v\in V_i.$ Finally, the relation 3) of Definition \ref{defalg} holds because $\theta_{i\bullet}(w)=0,$
when $w\in W_i.$ Thus, we obtain the homomorphism of DG algebras $\widebar{\xi}: \dR_{\mathfrak F}(\chi)\to \dE_{\mathfrak{F}}(\overline{\delta}).$ A direct check shows that this homomorphism is a quasi-isomorphism, i.e. the DG algebra $\dR_{\mathfrak F}(\chi)$
is isomorphic to the  cohomology  $H^{*}(\dE_{\mathfrak{F}}(\overline{\delta})).$
\end{proof}

\begin{theorem}\label{imbedR} There is a fully faithful functor $\prf\dR_{\mathfrak F}(\chi)\hookrightarrow \prf\dD_{\mathfrak{F}}(\overline{\delta}),$ which establishes an equivalence between the triangulated category $\prf\dR_{\mathfrak F}(\chi)$
and the left orthogonal
${}^{\perp}\cS,$ where $\cS\subset\prf\dD_{\mathfrak{F}}(\overline{\delta})$ is the full triangulated subcategory generated by the exceptional collection $\{\mK_1,\dots, \mK_m\}.$
\end{theorem}
\begin{proof}
By Proposition \ref{quasi-is}, there is a quasi-isomorphism of the DG algebras $\widebar{\xi}: \dR_{\mathfrak F}(\chi)\to \dE_{\mathfrak{F}}(\overline{\delta}).$ Therefore, the triangulated category $\prf\dR_{\mathfrak F}(\chi)$ is equivalent
to the triangulated subcategory of $\prf\dD_{\mathfrak{F}}(\overline{\delta})$ which is generated by $\mP_1$ and $\mP_2.$ By Proposition \ref{gencat}, this subcategory coincides with  the left orthogonal
${}^{\perp}\cS,$ where $\cS\subset\prf\dD_{\mathfrak{F}}(\overline{\delta})$ is the full triangulated subcategory generated by the exceptional collection $\{\mK_1,\dots, \mK_m\}.$
\end{proof}

\begin{corollary}
For any equidimensional family $\mathfrak{F}$ the dimension of the triangulated category $\prf\dR_{\mathfrak F}(\chi)$ does not exceed $2.$
\end{corollary}
\begin{proof}
The category $\prf\dD_{\mathfrak{F}}(\overline{\delta})$ has a full exceptional collection with three blocks $\mQ_{\bone},\; \mQ_{\btwo}$ and $\mQ_{\mathbf l_i},\; i=1\dots, m.$
Hence, the dimension of $\prf\dD_{\mathfrak{F}}(\overline{\delta})$ is not more than $2.$ The dimension of the admissible subcategory $\prf\dR_{\mathfrak F}(\chi)\subset\prf\dD_{\mathfrak{F}}(\overline{\delta})$ does not exceed the dimension of the whole triangulated category $\prf\dD_{\mathfrak{F}}(\overline{\delta}).$
\end{proof}

\section{Noncommutative resolutions of singular curves}

\subsection{Equidimensional case with  n=2}
In this section we consider a special case that is directly related to a geometry of singular rational curves. It will be  shown that some DG algebras of the form
$\dR_{\mathfrak{F}}(\chi)$ provide noncommutative smooth resolutions for  singular rational curves that can be constructed directly from the family $\mathfrak{F}.$
Let us recall a definition of a noncommutative smooth resolution.

\begin{definition} {\sf A noncommutative (or categorical) smooth resolution} of a scheme $\rY$ is a fully faithful functor
\[
\pi^*: \prf\rY\hookrightarrow \prf\dA,
\]
where $\dA$ is a cohomologically bounded DG algebra which is smooth.
\end{definition}
\begin{remark}
{\rm
If $\pi: \widetilde{\rY}\to \rY$  is a usual resolution of singularities of a singular variety $\rY$  that has rational singularities, then the functor  $\bL\pi^*: \prf\rY \to \prf\widetilde{\rY}$ delivers an example of a noncommutative smooth resolution of $\rY.$
Note, however, that if $\bR\pi_*\cO_{\widetilde{\rY}}\ncong \cO_{\rY},$ then the functor  $\bL\pi^*$ is not a noncommutative resolution of $\rY$ since the functor $\bL\pi^*$ is not fully faithful in this case.}
\end{remark}

Now, we consider a special case of the equidimensional family $\mathfrak{F}$ when $n,$ which is the dimension of the vector space $C,$ is equal to $2$ and all vector subspaces $V_i,\, W_i\subset C,\; i=1,\dots, m,$ are one-dimensional.
We also assume that $\delta_i=0$ for all $i=1,\dots, m.$ Thus, in this case we have a usual quiver $(\Phi_{\mathfrak{F}}, J_{\mathfrak F})$ with the following set of arrows:
\begin{flalign*}
\begin{array}{rl}
\mathrm{a1)} & \text{two arrows $\{c_1, c_2\}$ from $\bone$ to $\btwo,$} \\
\mathrm{a2)} & \text{two arrows $\beta_i$ and $\phi_i$  from $\btwo$ to ${\mathbf l_i}$ for each $i=1,\dots, m,$}\\
\end{array}&&
\end{flalign*}
while the ideal of relations $J_{\mathfrak{F}}$ is  generated by the elements:
\begin{flalign*}
\begin{array}{rlll}
\mathrm{r1)} & \beta_i v_i, & \text{where}\quad \langle v_i\rangle= V_i& \text{for any}\quad i=1,\dots, m,\\
\mathrm{r2)} & \phi_{i} w_i, & \text{where}\quad \langle w_i\rangle= W_i& \text{for any}\quad i=1,\dots, m.\\
\end{array}&&
\end{flalign*}

As above, denote by $\dD_{\mathfrak{F}}$ the quotient path algebra $\kk\Phi_{\mathfrak{F}}/J_{\mathfrak F}$ of the quiver with relations $(\Phi_{\mathfrak{F}}, J_{\mathfrak F}).$
We know that the projective $\dD_{\mathfrak{F}}$\!--modules $\mQ_{\bone}=e_{\bone}\dD_{\mathfrak{F}},$ $\mQ_{\btwo}=e_{\btwo}\dD_{\mathfrak{F}},$ and $\mQ_{\mathbf l_i}=e_{\mathbf l_i}\dD_{\mathfrak{F}},\; i=1\dots, m,$ form a full exceptional collection in the triangulated category $\prf\dD_{\mathfrak{F}}$
such that the objects $\mQ_{\mathbf l_i},\; i=1\dots, m,$ are completely orthogonal (see Proposition \ref{excol}). Moreover, in the case of a usual quiver this exceptional collection is strong and, by Proposition \ref{excolhom}, all nontrivial
$\Hom$\!--spaces are two-dimensional vector spaces (see Fig. \ref{quiver2}).

\begin{figure}[ht]
\hspace*{0cm}
\xymatrix{
&&&& \bullet\, {\mathbf l_{1}}\\
&&&& \bullet\, {\mathbf l_{2}}\\
\underset{\bone}{\bullet}
\ar@/^0.8pc/@<0.5ex>[urrrr]
\textcolor{red}{}
\ar@[red]@{-->}@/^0.8pc/@<-0.3ex>[urrrr]
\ar@/^1pc/@<0.9ex>[uurrrr]
\textcolor{red}{}
\ar@[red]@{-->}@/^1pc/@<0.1ex>[uurrrr]
\ar@/_0.8pc/@<0.3ex>[drrrr]
\textcolor{red}{}
\ar@[red]@{-->}@/_0.8pc/@<-0.5ex>[drrrr]
\ar@<0.5ex>[rr]
\ar@<-0.5ex>[rr] &&
\underset{\btwo}{\bullet}
\ar@<0.5ex>[uurr] |(0.47)\hole
\textcolor{red}{}
\ar@[red]@{-->}@<-0.3ex>[uurr] |(0.51)\hole
\ar@<-0.1ex>[urr]
\textcolor{red}{}
\ar@[red]@{-->}@<-0.9ex>[urr]
\ar@<0.9ex>[drr]
\textcolor{red}{}
\ar@[red]@{-->}@<0.1ex>[drr]&&
\genfrac{}{}{0pt}{}{\vdots}{\vdots}\\
&&&& \bullet\, {\mathbf l_{m}}}
\caption{The algebra $\dD_{\mathfrak{F}}$}
\label{quiver2}
\end{figure}

\subsection{Singular rational curves with ordinary double points}
First, it makes sense to consider a simpler case, when  the family $\mathfrak{F}$ consists of $2m$ different one-dimensional subspaces.
Let us consider the projective line $\PP^1=\PP^1(C^{\vee})$ and denote by $\rv_i, \rw_j\in \PP^1, \; 1\le i,j\le m,$ the different closed points on the projective line $\PP^1$ that are related
to the vector subspaces $V_i$ and $W_j,$ respectively. Gluing the points $\rv_i$ and $\rw_i$ to nodes for each $i=1,\dots, m$ we obtain a nodal curve $\rY$ with $m$ different nodal points $\rs_1,\dots, \rs_m\in \rY.$
The curve $\rY$ can be obtained as a fibred coproduct $\PP^1\coprod_{\rT} \rS$ in the category of schemes, where $\rT=\coprod_{i=1}^m (\rv_i\sqcup\rw_i)\subset \PP^1$ is the closed subscheme of $\PP^1,$ consisting of $2m$ closed points $\rv_i, \rw_j,$ and $\rS=\coprod_{i=1}^m \rs_i$ is the disjoint union of $m$ points. Thus, we have the following diagram of schemes
\begin{equation}\label{fiber}
\begin{split}
\xymatrix{
\PP^1 \ar[d]_{\rf} & \rT\ar[l]_{\rj}\ar[d]^{\rp}\ar[dl]_{\rt}\\
\rY & \rS\ar[l]_{\ri}
}
\end{split}
\end{equation}
which is both cartesian and cocartesian.

The triangulated category $\prf\PP^1=\D^b(\coh \PP^1)$ has a full exceptional collection $(\cO_{\PP^1}(-1), \cO_{\PP^1})$ and, hence, it is equivalent to the full triangulated subcategory $\cQ\subset\prf\dD_{\mathfrak{F}}$ which is
generated by the projective modules $\mQ_{\bone}$ and $\mQ_{\btwo}.$ Denote by ${\mathrm\Phi}:\prf\PP^1\to \prf\dD_{\mathfrak{F}}$ the fully faithful functor which sends $\cO_{\PP^1}(-1)$ and $\cO_{\PP^1}$ to
$\mQ_{\bone}$ and $\mQ_{\btwo},$ respectively. The left orthogonal ${}^{\perp}\cQ$ is generated by the projective $\dD_{\mathfrak{F}}$\!--modules $\mQ_{\mathbf l_i}, \; i=1,\dots, m,$ and it is equivalent to
the triangulated category $\prf\rS,$ where $\rS=\coprod_{i=1}^m \rs_i$ is the disjoint union of $m$ points,  as above. Denote by ${\mathrm\Psi}:\prf\rS\to \prf\dD_{\mathfrak{F}}$ the fully faithful functor which sends
the structure sheaf of the point $\rs_i$ to the projective module $\mQ_{\mathrm l_i}.$

It is useful to note that actually the DG category of perfect DG modules $\prfdg\dD_{\mathfrak{F}}$ is quasi-equivalent to a gluing of the DG categories $\prfdg\PP^1$ and $\prfdg\rS$ via  the DG bimodule
defined by the stricture sheaf of the scheme $\rT$ via the functors $\rj_*$ and $\rp_*$ (see, e.g., \cite{KL,Or16} for definition of gluing).
Thus, there are  DG functors between DG categories
\[
{\mPhi}:\prfdg\PP^1\lto \prfdg\dD_{\mathfrak{F}} \quad\text{and}\quad {\mPsi}:\prfdg\rS\lto \prfdg\dD_{\mathfrak{F}},
\]
which induce the homotopy fully faithful functors $\rPhi$ and $\rPsi$ between the triangulated categories.
As we mentioned above, the DG category $\prfdg\dD_{\mathfrak{F}}$ is quasi-equivalent to a gluing of the DG categories $\prfdg\PP^1$ and $\prfdg\rS$ via the DG bimodule
defined by the stricture sheaf of the scheme $\rT.$ In particular, for any objects $\mF\in\prfdg\PP^1$ and $\mG\in\prfdg\rS$ we have maps
\begin{equation}\label{iota}
\zeta: \dHom_{\prfdg\rT}(\mj^*\mF, \mp^*\mG)\lto \dHom_{\dD_{\mathfrak{F}}}(\mPhi\mF, \mPsi\mG)
\end{equation}
that are quasi-isomorphisms of complexes of vector spaces.

The pullback DG functors
\[
\mf^*: \prfdg \rY\lto \prfdg \PP^1,\quad\text{and}\quad \mi^*:\prfdg \rY\lto \prfdg\rS
\]
define a DG functor
\[
\mPi^*: \prfdg\rY\lto\prfdg\dD_{\mathfrak{F}}
\]
which sends an object $\mE\in\prfdg\rY$ to the object $\Cone (\mPhi\mf^*\mE\stackrel{\zeta_{\mE}}{\lto} \mPsi\mi^*\mE)[-1],$
where $\zeta_{E}$ is the image of $\id_{\mt^* \mE}\in \dHom_{\prfdg\rT}(\mj^*\mf^*\mE, \mp^*\mi^*\mE)$ under the map $\zeta$ from (\ref{iota}).
The DG functor $\mPi$ induces the homotopy functor
$
{\rPi}^*:\prf\rY\to\prf\dD_{\mathfrak{F}}
$
between the triangulated categories.
\begin{proposition}\label{resol}
The functor $
{\rPi}^*:\prf\rY\to\prf\dD_{\mathfrak{F}}
$
is fully faithful.
\end{proposition}
\begin{proof}
The proof is the same as the proof of Proposition 6.5 from \cite{KL} (see details there).
For any $\mE\in\prfdg\rY$ the object $\mPi\mE$ is the cone $\Cone (\mPhi\mf^*\mE\stackrel{\zeta_{\mE}}{\lto} \mPsi\mi^*\mE)[-1].$
Therefore, for any two objects $\mE_1, \mE_2\in\prfdg\rY $ the complex $\dHom_{\dD_{\mathfrak{F}}}(\mPi\mE_1, \mPi\mE_2)$ is equal to the complex
\[
\Cone\Bigl(\dHom_{\dD_{\mathfrak{F}}}(\mPhi\mf^*\mE_1, \mPhi\mf^*\mE_2)\oplus \dHom_{\dD_{\mathfrak{F}}}(\mPsi\mi^*\mE_1, \mPsi\mi^*\mE_2)\lto
\dHom_{\dD_{\mathfrak{F}}}(\mPhi\mf^*\mE_1, \mPsi\mi^*\mE_2) \Bigr) [-1].
\]
Since the functors $\Phi$ and $\Psi$ are fully faithful the complex $\dHom_{\dD_{\mathfrak{F}}}(\mPi\mE_1, \mPi\mE_2)$ is naturally quasi-isomorphic to the complex
\[
\Cone\Bigl(\dHom_{\prfdg\PP^1}(\mf^*\mE_1, \mf^*\mE_2)\oplus \dHom_{\prfdg\rS}(\mi^*\mE_1, \mi^*\mE_2)\lto
\dHom_{\prfdg\rT}(\mj^*\mf^*\mE_1, \mp^*\mi^*\mE_2)\Bigr)[-1].
\]
The last complex is naturally quasi-isomorphic to a complex which calculates the $\Hom$\!--spaces from the object $\mE_1$ to the object
\[
M=\Cone\left(\rf_*\rf^*\mE_2\oplus \ri_* \ri^*\mE_2\lto \rf_*\rj_*\rp^*\ri^*\mE_2)\right)[-1]
\]
in the triangulated category $\D^b(\coh\rY).$ On the other hand, the object $M$ is isomorphic to the object $\mE_2$ in the triangulated category $\D^b(\coh\rY).$
Indeed, if we consider a tensor product of  $\mE_2$ and the exact triangle
\[
\cO_{\rY}\lto \rf_*\cO_{\PP^1}\oplus \ri_*\cO_{\rS}\lto\rf_*\rj_*\cO_{\rT},
\]
in $\D^b(\coh\rY),$ then we obtain a natural isomorphism $\mE_2\stackrel{\sim}{\to} M.$ Thus, the canonical map
\[
{\rPi}^*:\Hom_{\prf\rY}(\mE_1, \mE_2)\lto \Hom_{\prf\dD_{\mathfrak{F}}}(\rPi^*\mE_1, \rPi^*\mE_2)
\]
is an isomorphism for any pair of objects $\mE_1, \mE_2.$
\end{proof}

The DG functor $\mPi^*: \prfdg\rY\to\prfdg\dD_{\mathfrak{F}}$ also induces a  functor ${\rPi}_*:\prf\dD_{\mathfrak{F}}\to \D^b(\coh\rY)$ that is the right adjoint to the functor ${\rPi}^*.$
It is easy to see that the functor ${\rPi}_*$ restricted to the full subcategory $\rPhi(\prf\PP^1)\subset \prf\dD_{\mathfrak{F}}$ is isomorphic to the direct image $\rf_*:\prf\PP^1\to \D^b(\coh\rY).$
In particular, the functor $\rPi_*$ sends the object $\mQ_{\bone}$ and $\mQ_{\btwo}$ to the coherent sheaves $\rf_*\cO_{\PP^1}(-1)$ and $\rf_*\cO_{\PP^1},$ respectively.
On the other hand, it can be checked that the restriction of ${\rPi}_*$ to the full subcategory $\rPsi(\prf\rS)\subset \prf\dD_{\mathfrak{F}}$ is isomorphic to the functor
$\Cone(\ri_*\to \ri_*\rp_*\rp^*).$ In particular, the right adjoint functor $\rPi_*$ sends the object $\mQ_{\mathbf l_i}$ to the skyscraper sheaf $\cO_{\rs_i}$ for any $i=1,\dots, m.$

\begin{proposition}\label{ortog}
The image of the functor $
{\mathrm\Pi}^*:\prf\rY\hookrightarrow\prf\dD_{\mathfrak{F}}
$
is contained in the left orthogonal
${}^{\perp}\cS,$ where $\cS\subset\prf\dD_{\mathfrak{F}}$ is the full triangulated subcategory generated by the objects $\mK_1,\dots, \mK_m.$
\end{proposition}
\begin{proof} It follows from the fact that $\rPi_*\mK_i\cong 0$ for any $i.$ Indeed, the objects $\mK_i$ are defined as
\[
\Tot\left(\mQ_{\bone}\stackrel{v_i}{\lto} \mQ_{\btwo}\stackrel{\beta_i}{\lto} \mQ_{\mathbf l_i}\right).
\]
Therefore, the objects $\rPi_*\mK_i\in \D^b(\coh \rY)$ are isomorphic to the complexes
$\rf_*\cO_{\PP^1}(-1){\to} \rf_*\cO_{\PP^1}{\to} \cO_{\rs_i},$ which are acyclic.
\end{proof}
Propositions \ref{resol}, \ref{ortog}, and Corollary \ref{imbedR} imply that the category $\prf\dR_{\mathfrak{F}}(\chi_0),$ where $\chi_0:\ZZ^{m+1}\to\ZZ$ is such that
$\chi(\varepsilon_0)=0$ and $\chi(\varepsilon_i)=1$ for every $i=1,\dots, m$ yields  a noncommutative smooth resolution of the singular curve $\rY$ with the Grothendieck group
$K_0(\prf\dR_{\mathfrak{F}}(\chi_0))=\ZZ^2.$

\begin{theorem}\label{nodres}
The DG functor $\mPi^*: \prfdg\rY\to\prfdg\dD_{\mathfrak{F}}$ induces a fully faithful functor
$\prf\rY\hookrightarrow\prf\dR_{\mathfrak{F}}(\chi_0),$ where $\chi_0:\ZZ^{m+1}\to\ZZ$ is such a homomorphism that
$\chi(\varepsilon_0)=0$ and $\chi(\varepsilon_i)=1$ for every $i=1,\dots, m.$
\end{theorem}

\begin{remark}
{\rm Such a minimal noncommutative resolution of a rational nodal curve can be generalized on the case of a nodal singular curve $\rY$ of arbitrary geometric genus $g.$
Let $\rf: \wt\rY\to \rY$ be a  usual resolution, where $\wt\rY$ is a smooth curve of genus $g.$ We can consider the commutative diagram (\ref{fiber}) with $\wt\rY$ instead $\PP^1$ and
take the gluing $\prfdg\wt\rY\underset{\rT}{\oright}\prfdg\rS$ of the DG categories $\prfdg\wt\rY$ and $\prfdg\rS$ via  the DG bimodule
defined by the stricture sheaf of the scheme $\rT.$ In the same way  we can define the sequence of objects $\{\mK_1,\dots, \mK_m\}$ in the category $\prfdg\wt\rY\underset{\rT}{\oright}\prfdg\rS.$ Using Theorem 1 from \cite{Or17} or Theorem 3.22 from \cite{Or18}, it is not so difficult to show that this sequence forms  an exceptional collection in the triangulated category of the gluing and  the left orthogonal
${}^{\perp}\cS$ to the subcategory  $\cS,$  generated by the objects $\mK_1,\dots, \mK_m,$ provides a minimal smooth noncommutative resolution of the singular curve $\rY.$
That is, we obtain a full embedding of the category $\prf\rY$ to the smooth proper category ${}^{\perp}\cS.$ The triangulated category ${}^{\perp}\cS$ is called a Krull--Schmidt partner to the triangulated category $\prf\wt\rY$ (see \cite{Or17, Or18}).
In particular, the whole K-theory of the triangulated category ${}^{\perp}\cS$ is isomorphic to the whole K-theory of the triangulated category $\prf\wt\rY.$
}
\end{remark}

\subsection{General case of equidimensional families with n=2}
Let us consider a general case of an equidimensional family $\mathfrak{F}$ with $n=2$ and $k=1.$ Thus, some subspaces among $2m$ subspaces $V_i,\, W_j$ may coincide.
However, we still assume that the property (G) holds, i.e. $V_i\ne W_j$ when $i\ge j.$ As above, we denote by $\rT$ and $\rS$ the disjoint union of $2m$ and $m$ points, respectively, and denote by $\rp:\rT\to \rS$ the natural double covering. The closed points $\rv_i, \rw_j\in \PP^1, \; 1\le i,j\le m,$  on the projective line $\PP^1$ define a morphism of schemes  $\rj:\rT\to \PP^1$ that is not necessarily  a closed embedding now.
Denote by $\rY$ the fiber coproduct $\PP^1\coprod_{\rT} \rS.$ It is exists as a singular curve and can be described as follows. Take an affine line $\AA^1\subset\PP^1$ which contains all points $\rv_i, \rw_j\in \PP^1.$
The coproduct $\AA^1\coprod_{\rT} \rS$ is an affine curve $\Spec B,$ where the algebra $B$ is the kernel of the map
\begin{equation}\label{lambda}
\xymatrix@1{
\lambda: \kk[z]\times \underbrace{\kk\times\dots\times\kk}_{m} \ar[rr]^(.58){(\rj^*, -\rp^*)} && \relax\underbrace{\kk\times\dots\times\kk}_{2m}
}
\end{equation}
 induced by
$\rj$ and $\rp.$ The curve $\rY$ is the natural one-point compactification of $\Spec B.$

With any such equidimensional family $\mathfrak{F}$ with $n=2,\, k=1$ we can associate an oriented graph $\varGamma_{\mathfrak{F}}.$ The vertices of $\varGamma_{\mathfrak{F}}$ are all one-dimensional subspaces of $C$ which appear as elements of the family $\mathfrak{F}.$ The graph $\varGamma_{\mathfrak{F}}$ has exactly $m$ arrows: namely, for one arrow
 from $W_i$ to $V_i$ for each $i=1,\dots, m.$ Thus, the connected components of the graph $\varGamma_{\mathfrak{F}}$ one-to-one correspond to the singular points of the curve $\rY,$ while the vertices of $\varGamma_{\mathfrak{F}}$ one-to-one correspond to the points on $\PP^1$ which form the preimage of the singular points of $\rY.$

\begin{definition}
An equidimensional family $\mathfrak{F}$ with $n=2,\; k=1$ satisfying property (G) will be called {\sf modest} if the graph $\varGamma_{\mathfrak{F}}$ is a forest, i.e. it does not have nonoriented cycles.
\end{definition}
\begin{lemma}
A family ${\mathfrak F}$ is modest if and only if the map $\lambda$ is surjective.
\end{lemma}
\begin{proof} It is easy to see that the map $\lambda$ is surjective iff the first cohomology group $H^1(\varGamma_{\mathfrak{F}})$ of the graph $\varGamma_{\mathfrak{F}}$ is trivial. This holds
if and only if $\varGamma_{\mathfrak{F}}$ is forest.
\end{proof}
\begin{remark}\label{remshort}
{\rm
Note that the surjectivity of the  map $\lambda$  is equivalent to the condition that the natural sequence of coherent sheaves on $\rY$
\[
\cO_{\rY}\lto \rf_*\cO_{\PP^1}\oplus \ri_*\cO_{\rS}\lto\rf_*\rj_*\cO_{\rT}
\]
is actually a short exact sequence.
 }
 \end{remark}

Now, we can consider the commutative diagram (\ref{fiber}) as before, despite the fact that  the morphisms $\rj$ and $\ri$ are no longer necessarily closed embeddings, and the diagram is not necessarily cartesian.
However, as above, the pullback DG functors
\[
\mf^*: \prfdg \rY\lto \prfdg \PP^1,\quad\text{and}\quad \mi^*:\prfdg \rY\lto \prfdg\rS
\]
define a DG functor
\[
\mPi^*: \prfdg\rY\lto\prfdg\dD_{\mathfrak{F}},
\]
which sends an object $\mE\in\prfdg\rY$ to the object $\Cone (\mPhi\mf^*\mE\stackrel{\zeta_{\mE}}{\lto} \mPsi\mi^*\mE)[-1],$
where $\zeta_{E}$ is the image of $\id_{\mt^* \mE}\in \dHom_{\prfdg\rT}(\mj^*\mf^*\mE, \mp^*\mi^*\mE)$ under the map $\zeta$ from (\ref{iota}).
The DG functor $\mPi$ induces the homotopy functor
$
{\rPi}^*:\prf\rY\to\prf\dD_{\mathfrak{F}}
$
between the triangulated categories.

Using the surjectivity of the map $\lambda$ (\ref{lambda}) in the form as in Remark \ref{remshort} and repeating the proofs of Propositions \ref{resol} and \ref{ortog}, we can obtain the following results.

\begin{proposition}\label{resol2}
For any modest family ${\mathfrak F}$  the functor $
{\rPi}^*:\prf\rY\to\prf\dD_{\mathfrak{F}}
$
is fully faithful and  the image of $
{\mathrm\Pi}^*
$
is contained in the left orthogonal
${}^{\perp}\cS,$ where $\cS\subset\prf\dD_{\mathfrak{F}}$ is the full triangulated subcategory generated by the exceptional collection $\{\mK_1,\dots, \mK_m\}.$
\end{proposition}

\begin{theorem}\label{genres}
For any modest family ${\mathfrak F}$  the DG functor $\mPi^*: \prfdg\rY\to\prfdg\dD_{\mathfrak{F}}$ induces a fully faithful functor
$\prf\rY\hookrightarrow\prf\dR_{\mathfrak{F}}(\chi_0),$ where $\chi_0:\ZZ^{m+1}\to\ZZ$ is such a homomorphism that
$\chi(\varepsilon_0)=0$ and $\chi(\varepsilon_i)=1$ for every $i=1,\dots, m.$
\end{theorem}

\begin{remark}
{\rm
Let us take an arbitrary equidimensional family $\mathfrak{F}$ with $n=2$ and $k=1.$ Consider the singular curve $\rY_{\mathfrak{F}}$ that is obtained as a coproduct $\PP^1\coprod_{\rT} \rS,$ where $\rT$ and $\rS$ are the disjoint union of $2m$ and $m$ points as above. If the family $\mathfrak{F}$ is not modest, the map $\lambda$ is not surjective, and we do not have a fully faithful functor
from $\prf\rY_{\mathfrak{F}}$ to $\prf\dD_{\mathfrak{F}}.$ On the other hand, in this case we can reduce the family $\mathfrak{F}$ to a smaller subfamily $\mathfrak{G}\subset \mathfrak{F},$ which is modest
and for which the curve $\rY_{\mathfrak{G}}$ coincides  with the curve $\rY_{\mathfrak{F}}.$
A subfamily $\mathfrak{G}$ can be obtained from $\mathfrak{F}$ by removing some pairs $(V_i, W_i)$ such that the resulting graph $\varGamma_{\mathfrak{G}}$ is a full (or maximal) spanning forest for $\varGamma_{\mathfrak{F}}.$
Obviously, this can always be done. Thus, for an arbitrary equidimensional family $\mathfrak{F}$ with $n=2$ and $k=1,$ the resulting curve $\rY_{\mathfrak{F}}$ has a noncommutative smooth resolution
$\prf\rY_{\mathfrak{F}}\hookrightarrow\prf\dR_{\mathfrak{G}}(\chi_0),$ where $\mathfrak{G}$ is a some subfamily of the family $\mathfrak{F}.$
}
\end{remark}

\subsection{Dimension of triangulated categories in geometric case}
Let us now calculate dimensions of the triangulated categories $\prf\dD_{\mathfrak{F}}(\overline{\delta})$ and $\prf\dR_{\mathfrak{F}}(\chi)$ in the case  of an equidimensional family $\mathfrak{F}$ with $n=2$ and $k=1.$ Using the methods of \cite[Sec. 6]{El}, one can show that the dimensions of the triangulated categories $\prf\dD_{\mathfrak{F}}(\overline{\delta})$ and $\prf\dR_{\mathfrak{F}}(\chi)$
are both equal to $1.$

\begin{theorem}\label{dimcat}
Suppose that the family $\mathfrak{F}$ is equidimensional with $n=2$ and $k=1.$ Then the dimensions of the triangulated categories $\prf\dD_{\mathfrak{F}}(\overline{\delta})$ and $\prf\dR_{\mathfrak{F}}(\chi)$
are both equal to $1.$
\end{theorem}
\begin{proof}
The dimension of the admissible subcategory $\prf\dR_{\mathfrak{F}}(\chi)\subset \prf\dD_{\mathfrak{F}}(\overline{\delta})$ does not exceed
the dimension of the whole category $\prf\dD_{\mathfrak{F}}(\overline{\delta})$ and, besides, it can not be $0.$ Hence, it is sufficient to prove that the dimension
of the triangulated category $\prf\dD_{\mathfrak{F}}(\overline{\delta})$ is equal to $1.$

To simplify notation and reasoning, we consider the most interesting case when $\overline{\delta}\equiv 0.$

Let us take the full admissible triangulated subcategory $\cQ\subset\prf\dD_{\mathfrak{F}}$ that is
generated by the projective modules $\mQ_{\bone}$ and $\mQ_{\btwo}.$
The triangulated subcategory $\cQ$ is equivalent to the category $\prf\PP^1=\D^b(\coh \PP^1).$
Denote by ${\mathrm\Phi}:\prf\PP^1\to \prf\dD_{\mathfrak{F}}$ the fully faithful functor which sends $\cO_{\PP^1}(-1)$ and $\cO_{\PP^1}$ to
$\mQ_{\bone}$ and $\mQ_{\btwo},$ respectively.
The left orthogonal ${}^{\perp}\cQ$ is generated by the projective $\dD_{\mathfrak{F}}$\!--modules $\mQ_{\mathbf l_i}, \; i=1,\dots, m$ and it is equivalent to
the triangulated category $\prf\rS,$ where $\rS=\coprod_{i=1}^m \rs_i$ is the disjoint union of $m$ points.

The DG category of perfect DG modules $\prfdg\dD_{\mathfrak{F}}$ is quasi-equivalent to a gluing of the DG categories $\prfdg\PP^1$ and $\prfdg\rS$ with respect to  the DG bimodule
defined by the stricture sheaf of the scheme $\rT$ via the functors $\rj_*$ and $\rp_*,$ where $\rj$ and $\rp$ are the morphisms in commutative square (\ref{fiber}).
This means that for any objects $E\in\cQ\cong\prf\PP^1$ and $F\in{}^{\perp}\cQ\cong\prf\rS,$ we have
\[
\Hom_{\prf\dD_{\mathfrak{F}}}(E,F)\cong\Hom_{\rT}(\rj^*E,\rp^*F)\cong\Hom_{\PP^1}(E, \rj_*\rp^*F).
\]
Any object $\rj_*\rp^*F$ is a direct sum of the objects of the form $\cO_x[l],$ where the closed point $x\in\PP^1$ is in the image of the map $\rj.$
Denote by $\cX$ the finite set of all such closed points and by $c$ the cardinality of the set $\cX.$ Note that $c\le 2m.$

First, we will show that for any object $E\in \prf\PP^1$ there is an exact triangle
\begin{equation}\label{extr}
s(E)\stackrel{\varsigma}{\lto} E\stackrel{\varpi}{\lto} q(E)\stackrel{\vartheta}{\lto} s(E)[1]
\end{equation}
in the triangulated category $\prf\PP^1$ such that
\begin{itemize}
\item[(1)] $s(E)\in \langle \cO_{\PP^1}\rangle;$
\item[(2)] all irreducible summands of the object $q(E)$ belong (up to shift) to the finite set of
coherent sheaves on $\PP^1$ consisting of $\cO_x,\;x\in\cX$ and the line bundles $\{ \cO_{\PP^1}(-1),\ldots, \cO_{\PP^1}(c)\};$
\item[(3)] for any morphism $f:E\to \rj_*\cO_{\rT}[l]$ one has $f\varsigma=0.$
\end{itemize}

It is sufficient to construct such an exact triangle for every indecomposable coherent sheaves on $\PP^1,$ i.e. for each line bundles $\cO(k)$ and any torsion sheaf
$\cO_{ly},$ where $y\in \PP^1$ is a closed point and $l$ is a positive integer.

In the case when $E$ is  a line bundle $\cO_{\PP^1}(-k),\; k\ge 1,$ we put $q(E)=\cO_{\PP^1}(-1)^{k}$ and $s(E)=\cO_{\PP^1}^{k-1}[-1].$ If $E=\cO_{\PP^1}(k),\; k=0,\dots, c,$ we take $s(E)=0,\;q(E)=E.$

Finally, for any $E=\cO_{\PP^1}(k),\;k>c,$ using the exact sequence
\[
0\lto\cO_{\PP^1}(-1)^{k-c}\lto\cO_{\PP^1}^{k+1-c}\lto\cO_{\PP^1}(k)\lto \cO_{\cX}\lto 0,
\]
we put
$s(E)=\cO_{\PP^1}^{k+1-c}$ and $q(E)= \cO_{\cX}\oplus\cO_{\PP^1}(-1)^{k-c}[1].$

In the case $E=\cO_{ly}$ when $y\not\in\cX ,$ we put $s(E)=\cO_{\PP^1}^l,\; q(E)=\cO_{\PP^1}(-1)^l[1].$
If $E=\cO_{lx}$ and $x\in\cX,$ we can consider the following exact sequence
\[
0\lto\cO_{\PP^1}(-1)^{l-1}\lto\cO_{\PP^1}^{l-1}\lto\cO_{lx}\lto \cO_x\lto 0.
\]
Using this exact sequence, we put $s(E)=\cO_{\PP^1}^{l-1}$ and $q(E)= \cO_x\oplus\cO_{\PP^1}(-1)^{l-1}[1].$

Now, we a ready to show that the set $\K=\{ {\mathrm\Phi}\cO_{\PP^1}(-1),\ldots, {\mathrm\Phi}\cO_{\PP^1}(c);\; {\mathrm\Phi}\cO_x,\;x\in\cX;\; \mQ_{\mathbf l_i}, \; i=1,\dots, m\} $
generates the triangulated category $\prf\dD_{\mathfrak{F}}$ at one step.
Any object of the category $\prf\dD_{\mathfrak{F}}$ can be obtained as a cone of a map $f:E\to F,$ where $E\in\cQ\cong\prf\PP^1$ and $F\in {}^{\perp}\cQ.$
Consider exact triangle (\ref{extr}) for the object $E.$ Since $f\varsigma=0,$ the morphism $f$ factors through $q(E)$ as $f'\varpi,$ where $f'$ is a morphism
from $q(E)$ to $F.$ Now the object $\Cone(f)$ is isomorphic to the object  $\Cone(q(E)\stackrel{(\vartheta, f')}{\lto}s(E)[1]\oplus F).$
It remains only to note that all irreducible summands of the objects $q(E),\, s(E),\, F$ belong (up to a shift) to the finite set $\K.$
Therefore, the set $\K$ generates the whole category $\prf\dD_{\mathfrak{F}}$ at one step.
\end{proof}

\begin{conclusion}
{\rm Thus, in the case of an equidimensional family $\mathfrak{F}$   with $n=2$ and $k=1,$ the DG category $\prfdg\dR_{\mathfrak{F}}(\chi)$ and its triangulated category $\prf\dR_{\mathfrak{F}}(\chi)$
can be considered as a smooth proper noncommutative curve, because the category $\prf\dR_{\mathfrak{F}}(\chi)$ has the dimension equal to $1.$
Moreover, for  modest families $\mathfrak{F}$ and  $\chi_0:\ZZ^{m+1}\to\ZZ$ with
$\chi(\varepsilon_0)=0$ and $\chi(\varepsilon_i)=1$ for $i=1,\dots, m,$  the categories $\prf\dR_{\mathfrak{F}}(\chi_0)$  provide smooth noncommutative resolutions for singular rational curves $\rY_{\mathfrak{F}}$
and such noncommutative resolutions are minimal in sense that the whole K-theory of the category $\prf\dR_{\mathfrak{F}}(\chi_0)$ is isomorphic to the direct sum $K_*(\kk)\oplus K_*(\kk)$ of two copies of the K-theory of the field.

}
\end{conclusion}


\end{document}